\documentclass[11pt,a4paper,reqno]{amsart}
\usepackage{amsmath}
\usepackage{amsfonts}
\usepackage{amssymb}
\usepackage{amsthm}
\usepackage{newlfont}
\usepackage{a4}
\usepackage{enumerate}
\usepackage[figuresright]{rotating}
\usepackage{color}

\theoremstyle{definition}
\newtheorem{defn}{Definition}[section]

\newtheorem{tvr}[defn]{Proposition}

\theoremstyle{remark}

\usepackage[centertags]{amsmath}
\usepackage{graphicx}
\usepackage{amsfonts}
\usepackage{amssymb}
\usepackage{amsthm}
\usepackage{bbm}
\usepackage{newlfont}
\usepackage{epsfig,longtable}

\newlength{\defbaselineskip}
\newcommand{\setlinespacing}[1]%
           {\setlength{\baselineskip}{#1 \defbaselineskip}}

\newcommand{\ep}{\varepsilon}
\newcommand{\bC}{\mathbb{C}}
\newcommand{\bN}{\mathbb{N}}
\newcommand{\bZ}{\mathbb{Z}}
\newcommand{\mL}{\mathcal{L}}

\newcommand{\mG}{\mathcal{G}}
\newcommand{\mA}{\mathcal{A}}

\newcommand{\mS}{\mathcal{S}}

\newcommand{\I}{\mathcal{I}}

\newcommand{\sS}{S}
\newcommand{\sC}{C}

\newcommand{\p}{\cdot}

\newcommand{\Span}{\operatorname{span}}

\newcommand{\gl}{\operatorname{gl}}
\newcommand{\spl}{\operatorname{sl}}
\newcommand{\slc}{\operatorname{sl}(3,\bC)}
\newcommand{\sla}{\operatorname{sl}(2,\bC)}

\newcommand{\der}{\operatorname{der}}
\newcommand{\ad}{\operatorname{ad}}

\newcommand{\Tr}{\operatorname{Tr}}

\newcommand{\Id}{\operatorname{Id}}

\newcommand{\wt}{\widetilde}
\newcommand{\ev}{\varepsilon}
\newcommand{\Ga}{\Gamma}

\newcommand{\xor}{\operatorname{xor}}
\newcommand{\dima}{\dim_{(\alpha,\beta,\gamma)}}

\newcommand{\set}[2]{\left\{ #1 \, |\, #2 \right\}}

\newcommand{\bul}{-}
\newcommand{\ci}{+}

\newcount\minute	
\newcount\hour		
\newcount\hourMins  
\def\now
{\minute=\time
\hour=\time \divide \hour by 60
\hourMins=\hour \multiply\hourMins by 60
\advance\minute by -\hourMins
\zeroPadTwo{\the\hour}:\zeroPadTwo{\the\minute}}

\def\todays
{\zeroPadTwo{\the\day}.\,\zeroPadTwo{\the\month}.\,\the\year}
\def\zeroPadTwo#1%
{\ifnum #1<10 0\fi #1}

\begin{document}

\title[Graded contractions of the Gell-Mann graded $\slc$]
{Graded contractions of the Gell-Mann graded $\slc$}

\author[J. Hrivn\'{a}k]{Ji\v{r}\'{i} Hrivn\'{a}k$^{1}$}
\author[P. Novotn\'{y}]{Petr Novotn\'{y}$^1$}

\date{ \today}

\begin{abstract}
\small
The Gell-Mann grading, one of the four gradings of $\slc$ that cannot be further refined, is considered as the initial grading for the graded contraction procedure. Using the symmetries of the Gell-Mann grading, the system of contraction equations is reduced and solved. Each non--trivial solution of this system determines a Lie algebra which is not isomorphic to the original algebra $\slc$.  
The resulting $53$ contracted algebras are divided into two classes --- the first is represented by the algebras which are also continuous In\"on\"u--Wigner contractions, the second is formed by the discrete graded contractions.   
\end{abstract}

\maketitle
{\small

\medskip


\medskip

\noindent
$^1$ Department of physics,
Faculty of nuclear sciences and physical engineering, Czech
Technical University in Prague, B\v{r}ehov\'a~7, 115 19 Prague 1, Czech
Republic

\medskip
\noindent
\textit{E-mail:} jiri.hrivnak@fjfi.cvut.cz, petr.novotny@fjfi.cvut.cz
}
\section{Introduction}

This paper extends the work undertaken in \cite{PC4,HN1} where the graded contractions of the Cartan and Pauli graded $\slc$ are classified. In this article, the Gell-Mann matrices and their corresponding grading \cite{HPP3,PZ2} of $\slc$ are considered as a starting point for the graded contraction procedure. 
The Gell-Mann matrices represent a hermitian generalization of the two--dimensional Pauli matrices \cite{Georgi}.
The results of this procedure are  Gell-Mann graded complex Lie algebras, with characteristics specifically linked to the original algebra $\slc$.

The graded contraction procedure is introduced in \cite{PC1} as an alternative to continuous
In\"on\"u--Wigner (IW) contractions. In contrast to continuous contractions, graded contractions
represent an algebraical perspective --- instead of a limit of
sequences of mutually isomorphic Lie algebras,
the commutation relations among graded subspaces are multiplied by
contraction parameters. These parameters then have to satisfy the system of equations resulting from the Jacobi identities. The complexity of this contraction system depends on the number of
grading subspaces and the structure of a given grading. 

The graded contraction procedure of Lie algebras was applied to a number of cases and was also generalized to other graded algebras --- graded
contractions of inhomogeneous algebras \cite{dAzca}, central
extensions \cite{Guise}, affine algebras \cite{Hus,dM2}, Jordan
algebras \cite{KP}, Virasoro algebras \cite{Kostyak}, Lie algebra
$\operatorname{e}(2,1)$ \cite{PC5} were studied. Physically motivated cases
related to the kinematical and conformal group of space--time were
presented in \cite{dM1,T2}. 
Partial general results, for so called generic cases of gradings of Lie algebras, were constructed in \cite{WW4,WW5}. This approach solved the contraction systems simultaneously for all Lie algebras which allow a given grading.
However, it appears that in the generality of this approach, a significant part of the solutions would be omitted for non-generic cases. As the Gell-Mann grading of $\slc$ cannot be considered a generic case, the detailed analysis of possible outcome of graded contraction procedure is thus at present necessary.

The results of the graded contraction procedure fall into two distinct categories --- continuous graded contractions, which correspond to IW contractions, and discrete
graded contractions. In contrast to the IW contraction, a discrete graded contraction may yield a 
continuous parametric family of non--isomorphic Lie algebras --- such continuous parametric
families can be used in the deformation theory of Lie algebras. Additional advantages of continuous graded contractions are consequences of them also being IW contractions. Besides the  physical application of IW contractions --- studying the possibility of existence of a correspondence principle --- they can be used for obtaining invariants and representations of the contracted algebras \cite{IW}.

The concept of the graded contractions can be also used for
contractions of representations \cite{PC2,PN2}. The process of obtaining representations of the contracted algebras is not completely straightforward, the grading and the representation have to satisfy the so called compatibility condition. Even though the compatibility condition is always satisfied for root space decompositions, i.e. Cartan gradings, in other cases it may be a more challenging task \cite{Leng,PT,T1}. The general representation theory which exists
for simple Lie algebras, can be then used for the construction of
representations of other types of Lie algebras, especially solvable ones. The method is mostly valuable for discrete graded contractions  --- then the representations cannot be obtained by the standard way using IW contractions.

Thus, the goals of this work are: 
\begin{enumerate}[(i)]
 \item to obtain all IW contractions which start from the Gell-Mann grading of $\slc$ and preserve the Cartan subalgebra; classifying these contractions is motivated by their direct applicability for obtaining Casimir operators and representations, 
\item to obtain discrete graded contractions -- these may yield new representations of the contracted algebras in cases where the direct IW method is not possible,
\item to contribute to the classification problem of Lie algebras 
--- solvable Lie algebras are classified only for dimensions not
greater than six.
\end{enumerate}

The paper is organized as follows. In Section 2, the Gell-Mann grading and its symmetries are recalled. In Section 3, the graded contraction procedure is formulated and a suitable equivalence of the solutions introduced. The contraction system, its reduction by the symmetries and the solution, together with the set of higher-order identities is located in Section 4. The procedure, which is necessary for identification of the results, is applied to the contracted Lie algebras in Section 5. Concluding remarks and follow-up questions are contained in Section 6. In Appendix A.1, the contraction matrices are located, Appendix A.2 contains the final classification of graded contractions of the Gell-Mann graded $\slc$ in a tabulated form. Appendix A.3 contains the invariant functions of the one--parametric graded contractions. 

\section{ The Gell--Mann grading}

\subsection{The grading $\Ga$}\

Throughout this article, we use the following basis of the Lie algebra of three-dimensional complex traceless matrices $\spl(3,\bC)$:
\begin{equation*}
\begin{array}{llll}
e_1=\begin{pmatrix}
        1&0&0\\
        0&-1&0\\
        0&0&0
        \end{pmatrix}, &
e_2=\begin{pmatrix}
        0&0&0\\
        0&1&0\\
        0&0&-1
        \end{pmatrix}, &
e_3=\begin{pmatrix}
        0&1&0\\
        1&0&0\\
        0&0&0
        \end{pmatrix}, &
e_4=\begin{pmatrix}
        0&0&1\\
        0&0&0\\
        1&0&0\\
        \end{pmatrix},
\\[24pt]
e_5=\begin{pmatrix}
        0&0&0\\
        0&0&1\\
        0&1&0
        \end{pmatrix}, &
e_6=\begin{pmatrix}
        0&-1&0\\
        1&0&0\\
        0&0&0
        \end{pmatrix}, &
e_7=\begin{pmatrix}
        0&0&0\\
        0&0&-1\\
        0&1&0
        \end{pmatrix}, &
e_8=\begin{pmatrix}
        0&0&-1\\
        0&0&0\\
        1&0&0
        \end{pmatrix}.
\end{array}
\end{equation*}
The corresponding non--zero commutation relations $[e_i,e_j]=e_ie_j-e_je_i$, $i,j\in \{1,2,\dots,8\}$ of $\spl(3,\bC)$ written in this
basis are
\begin{equation}\label{comrel}
\arraycolsep=2pt
\begin{array}{llllll}
[e_1,e_3]=-2e_6, & [e_1,e_4]=-e_8,& [e_1,e_5]=e_7,& [e_1,e_6]=-2e_3,& [e_1,e_7]=e_5,& [e_1,e_8]=-e_4, \\[2pt] [e_2,e_3]=e_6, & [e_2,e_4]=-e_8,& [e_2,e_5]=-2e_7,& [e_2,e_6]=e_3,& [e_2,e_7]=-2e_5,& [e_1,e_8]=-e_4,\\[2pt]
[e_3,e_4]=-e_7,& [e_3,e_5]=-e_8,& [e_3,e_6]=2e_1,& [e_3,e_7]=-e_4,& [e_3,e_8]=-e_5,\\[2pt]
[e_4,e_5]=-e_6,& [e_4,e_6]=-e_5,& [e_4,e_7]=e_3,& \multicolumn{2}{l}{[e_4,e_8]=2(e_1+e_2),}\\[2pt]
[e_5,e_6]=e_4,& [e_5,e_7]=2e_2,& [e_5,e_8]=e_3,& [e_6,e_7]=-e_8,& [e_6,e_8]=e_7,& [e_7,e_8]=-e_6.
\end{array}
\end{equation}
The two-dimensional subspace $\Span_\bC\{e_1,e_2\} $ forms a Cartan subalgebra of $\spl(3,\bC)$. In the following, it is convenient to 
take into account ordered triplets $(i_1,i_2,i_3)$ with $i_1,i_2,i_3\in\{0,1\}$ which form the additive abelian group $\bZ_2^3$ with addition mod $2$ and introduce a set of seven indices $I$ such that 
$$I=\bZ_2^3\setminus \{(0,0,0)\}.$$
Let us denote the Cartan subalgebra by 
$$ L_{001}= \Span_\bC\{e_1,e_2\} $$
and the one-dimensional subspaces corresponding to the remaining basis vectors by
\begin{equation*}
\begin{alignedat}{3}
L_{111}&= \bC e_3,\quad & L_{101}&= \bC e_4,\quad & L_{011}&=\bC e_5, \\
L_{110}&= \bC e_6,\quad & L_{010}&= \bC e_7,\quad & L_{100}&=\bC e_8. 
\end{alignedat}
\end{equation*}
If we assign $L_{000}=\{0\}$, we obtain a decomposition $\Ga$ of the vector space $\spl(3,\bC)$ into the direct sum of the subspaces
\begin{align}\nonumber\label{ord_gell}
\Ga:\quad \spl(3,\bC) &= L_{001}\oplus L_{111}\oplus L_{101} \oplus
L_{011}\oplus L_{110} \oplus L_{010}\oplus L_{100}.
\end{align}
with the grading property 
$$[L_j,L_k]\subseteq L_{j+k}, \quad j,k\in I.$$

The grading $\Ga$ is commonly known as the orthogonal grading of $\spl(3,\bC)$ and is described in \cite{Kostrikin} as a special case of orthogonal gradings of $\spl(n,\bC)$. Since the grading subspaces of $\Ga$ are minimal and its index set is embedded in $\bZ_2^3$, the grading $\Ga$ forms so called fine group grading.
Bases of the grading subspaces of $\Ga$ are  formed by the Gell-Mann matrices \cite{GellMann} and therefore, we also refer to $\Ga$ as to the Gell-Mann grading.

\subsection{The symmetry group of $\Ga$}\

The symmetry group of the grading $\Ga$ consists of such automorphisms of $\spl (3, \bC)$ which permute the grading subspaces. For the case of the Gell-Mann grading, the induced permutation group can be realized as a certain finite matrix group which acts on the index set $I$. Such a matrix group, realizing permutations of $\Ga$ which correspond to automorphisms, is described in \cite{HPPT1,HPPT3}. It can be realized as the stability
subgroup $G$ of the point $(0,0,1)$ in the finite matrix group
$SL(3,\bZ_2)$, i.e.
\begin{equation*}
G= \left\{\left(
\begin{array}{ccc}
               a&b&e\\
               c&d&f\\
               0&0&1\\
              \end{array}
\right)\ \bigg|\  a,b,c,d,e,f\in\mathbb{Z}_2,\  ad-bc= 1  \pmod 2
\right\}.
\end{equation*}
The matrix group $G$ has 24 elements and its action on the index set $I$ is given as the right matrix multiplication,
i.e. for $A=\left(\begin{smallmatrix} a & b & e \\ c & d & f
\\ 0 & 0 & 1 \\
\end{smallmatrix}\right) \in G$ and $i=(i_1,i_2,i_3)\in I$ it
holds
\begin{equation*}
i  \mapsto iA = ((i_1a+i_2c)_{\hspace{-6pt}\mod2}, (i_1b+i_2d)_{\hspace{-6pt}\mod2},
(i_1e+i_2f+i_3)_{\hspace{-6pt}\mod2}).
\end{equation*}
In the following, the symmetry matrix group $G$ is crucial for construction and solving of the contraction system.

\section{ The Gell--Mann graded contractions}

\subsection{The graded contractions of $\Ga$}\

The graded contraction procedure for the grading $\Ga$ constructs new Lie algebras by introducing complex parameters $\ev_{ij}\in \bC$
and defining contracted commutation relations of grading subspaces $L_i$ via 
$$[x_i,x_j]_\ev = \ev_{ij}[x_i,x_j],$$
for all $x_i\in L_i$, $x_j\in L_j$ and $i,j\in I$. If for a pair of the subspaces $[L_i,L_j]=0$ holds then the corresponding contraction parameter $\ev_{ij}$ is irrelevant and we put $\ev_{ij}=0$.
Relevant are only such contraction parameters $\ev_{ij}$ for which 
$[L_i,L_j]\neq 0$ --- these have to fulfill the following conditions in order to  $[.\,,.]_\ev$ become a Lie bracket.

Firstly, the relevant contraction parameters have to satisfy the symmetry condition which corresponds to the required antisymmetry of $[.\,,.]_\ev$: 
$$\ev_{ij}=\ev_{ji}.$$
Secondly, the equation $e_{(i\,j\,k)}$ corresponding to Jacobi identity has to be satisfied for all $i,j,k\in I$:
\begin{equation}\label{rovnice}
\begin{array}{ll}
e_{(i\,j\,k)}: & \ev_{jk}\ev_{i,j+ k}[x_i,[x_j,x_k]] +
\ev_{ki}\ev_{j,k+i}[x_j,[x_k,x_i]] + \ev_{ij}\ev_{k,i+
j}[x_k,[x_i,x_j]] = 0 \\[4pt]
& \forall x_i\in L_i, \forall x_j\in L_j, \forall x_k\in L_k. \\
\end{array}
\end{equation}
A contraction equation $e_{(i\,j\,k)}$ as well as a contraction parameter $\ev_{ij}$
do not depend on the order of indices $i,j,k$, therefore we label contraction equations by unordered triplets of
contraction indices --- multisets of cardinality 3 --- and contraction parameters by  unordered pairs of grading indices. The set of all unordered triplets is denoted by $I^3_u$ and the set of all unordered pairs of grading indices is denoted by $I^2_u$. We collect all contraction equations in the set $\sS_\Ga$, i.e.
\begin{equation}\label{general}
\sS_\Ga=\set{e_w}{w\in I^3_u}.
\end{equation}

In general, the equations \eqref{rovnice} have three terms. In some cases of gradings, the equation \eqref{rovnice} reduces to two two-term equations. There are two approaches directed to reducing the equations \eqref{rovnice} to two terms while handling all Lie algebras with the same grading properties simultaneously.
\begin{enumerate}[(i)]
\item In the so called generic case, all contraction parameters are considered to be relevant~\cite{WW4}. Moreover, parameters corresponding to the space $L_{000}=\{0\}$ are defined. This approach leads to the system of 224 two--term equations
\begin{equation}\label{extend}
\ev_{jk}\ev_{i,j+ k} = \ev_{ki}\ev_{j,k+i} = \ev_{ij}\ev_{k,i+
j}, \qquad  i,j,k\in \bZ_2^3.
\end{equation}
\item The second less restrictive approach from~\cite{PC1} takes the system of contraction equations in the form 
\begin{equation}\label{ropa}
\ev_{jk}\ev_{i,j+ k} = \ev_{ki}\ev_{j,k+i} = \ev_{ij}\ev_{k,i+j} \qquad  i,j,k\in I,
\end{equation}
where all equations containing irrelevant parameters are omitted. This leads to the system consisting of 84 equations.
\end{enumerate}
 
The existence of general conditions for equivalence of the system~\eqref{rovnice} and the reduced systems~\eqref{extend},~\eqref{ropa} is still an open problem. Let us note that the systems ~\eqref{general} and~\eqref{ropa} are equivalent in the case of the Pauli grading~\cite{HN1} and the Cartan grading~\cite{PC4} of $\spl (3, \bC)$. On the contrary, we will see later that --- in our case of the Gell--Mann grading --- no two of these three systems are equivalent.

\subsection{Equivalence of solutions of $\sS_\Ga$}\

Contraction parameters are usually written in the form of a symmetric
square matrix $\ev = (\ev_{i,j})$ called a contraction matrix. The set of all contraction matrices, which solve the system $\sS_\Ga$, is denoted by $\sC_\Ga$. Each solution from $\sC_\Ga$ then determines a Lie algebra called graded contraction of the Gell-Mann graded $\spl (3, \bC)$. 
An important example of a trivial graded contraction is given by a normalization matrix $\alpha=(\alpha_{ij})$,
where
\begin{equation}\label{normat}
\alpha_{ij}=\frac{a_i a_j}{a_{i+j}}, \qquad a_k \in \bC \setminus \{0\},\, k \in I.
\end{equation}

We define an equivalence on the set of all contraction matrices in the following way. Two contraction matrices $\ev,\wt\ev \in \sC_\Ga$ are equivalent ($\ev \sim \wt\ev$) if there exist a normalization matrix $\alpha$ and $A\in G$ such that
\begin{equation}\label{sequiv}
\ev_{ij} = \frac{a_i a_j}{a_{i+j}}\wt\ev_{iA,jA}
\end{equation}
If $A = \Id$ then $\ev,\wt\ev$ are strongly equivalent ($\ev \approx \wt\ev$). It is shown in \cite{HN1} that two graded contractions given by two equivalent contraction matrices are isomorphic as Lie algebras.
Moreover, taking any $A \in G$ and considering an action
\begin{equation}\label{subst}
\ev_{ij} \mapsto \ev_{iA,jA}
\end{equation}
it can be easily verified (see \cite{HN1}) that the set of solutions $\sC_\Ga$ is invariant with respect to this action. For this reason we determine the orbits of this action on relevant contraction parameters. 

Taking the set of unordered pairs of indices $I^2_u$, this set splits
under the action $(i\, j) \mapsto (iA,\, jA)$ of $A\in G$ into five orbits. Two of these orbits correspond to
irrelevant contraction parameters: 6--point orbit represented by
$((0,1,0)(0,1,0))$ and 1--point orbit represented by
$((0,0,1)(0,0,1))$. The remaining three orbits are 
\begin{itemize}\itemsep 0pt
\item 12--point orbit represented by $((0,1,0)(1,0,0))$,
\item 6--point orbit represented by $((0,1,0)(0,0,1))$,
corresponding relevant contraction parameters will be marked by
superscript $\ci$,
\item 3--point orbit represented by $((0,1,0)(0,1,1))$,
corresponding relevant contraction parameters will be marked by
superscript $\bul$.
\end{itemize}
The set of pairs of indices formed by these three orbits corresponds to relevant contraction parameters and we denote it by $\I \subset I^2_u$. 
The general explicit form of the contraction matrix $\ev$, with irrelevant parameters on the diagonal set to zero
, is
\begin{equation*}
\renewcommand{\arraystretch}{1.2}
\ev=
 \begin{pmatrix}
   0 & \ev^\ci_{(001)(111)} & \ev^\ci_{(001)(101)} & \ev^\ci_{(001)(011)} &
   \ev^\ci_{(001)(110)} & \ev^\ci_{(001)(010)} & \ev^\ci_{(001)(100)} \\
\ev^\ci_{(001)(111)} & 0 & \ev_{(111)(101)} & \ev_{(111)(011)} &
\ev^\bul_{(111)(110)} & \ev_{(111)(010)} & \ev_{(111)(100)} \\
   \ev^\ci_{(001)(101)} & \ev_{(111)(101)} & 0 & \ev_{(101)(011)} &
   \ev_{(101)(110)} & \ev_{(101)(010)} & \ev^\bul_{(101)(100)} \\
\ev^\ci_{(001)(011)} & \ev_{(111)(011)} &  \ev_{(101)(011)} &0&
\ev_{(011)(110)} & \ev^\bul_{(011)(010)} & \ev_{(011)(100)} \\
   \ev^\ci_{(001)(110)} & \ev^\bul_{(111)(110)} & \ev_{(101)(110)} & \ev_{(011)(110)} &
   0 & \ev_{(110)(010)} & \ev_{(110)(100)} \\
\ev^\ci_{(001)(010)} & \ev_{(111)(010)} & \ev_{(101)(010)} &
\ev^\bul_{(011)(010)} & \ev_{(110)(010)} & 0 & \ev_{(010)(100)} \\
   \ev^\ci_{(001)(100)} & \ev_{(111)(100)} & \ev^\bul_{(101)(100)} & \ev_{(011)(100)} &
   \ev_{(110)(100)} & \ev_{(010)(100)} & 0 \\
 \end{pmatrix}.
\end{equation*}

\section{System of contraction equations}

\subsection{Symmetries and reduction of the system $\sS_\Ga$}\

It is shown in \cite{HN1} that similarly to the invariance of the set of contraction matrices $\sC_\Ga$ under the action of the symmetry group $G$, the set of contraction equations $\sS_\Ga$ is invariant under the action of $A\in G$
\begin{equation*}
e_{(i\, j\, k)}\mapsto e_{(iA,\, jA,\, kA)}.
\end{equation*}

The set $I^3_u$ of unordered triplets of grading indices, which label the equations from $\sS_\Ga$, is
decomposed into $11$ orbits with respect to the corresponding action on indices $(i\, j\, k)\mapsto(iA,\, jA,\, kA)$. The representatives of these orbits together with the corresponding number of their elements are summarized in Table~\ref{orb_gel3}. 
\begin{table}
\begin{center}
\renewcommand{\arraystretch}{1.5}
\begin{tabular}{|l|c||l|c||l|c|}
\hline
\multicolumn{4}{|c||}{Trivial equations} & \multicolumn{2}{|c|}{Non-trivial equations}  \\
\hline \hline
(0,1,0)(0,1,0)(1,0,0) & 24 & (0,1,0)(0,0,1)(0,0,1) & 6 & (0,1,0)(1,0,0)(0,1,1) & 12\\
(0,1,0)(0,1,0)(0,1,0) & 6  & (0,1,0)(1,0,0)(1,1,0) & 4 & (0,1,0)(1,0,0)(0,0,1) & 12\\
(0,1,0)(0,1,0)(0,1,1) & 6  & (0,1,0)(0,0,1)(0,1,1) & 3 & (0,1,0)(1,0,0)(1,1,1) & 4\\
(0,1,0)(0,1,0)(0,0,1) & 6  & (0,0,1)(0,0,1)(0,0,1) & 1 & & \\
\hline
\end{tabular}
\end{center}
\medskip
\caption{The representative elements of orbits of $I^3_u$ and the numbers of their points under
the action of the symmetry group $G$.}\label{orb_gel3}
\end{table}
Since all grading subspaces of the Gell-Mann grading form commutative subalgebras of $\spl(3,\bC)$, only the three orbits
in the last column in Table~\ref{orb_gel3} lead to non-trivial contraction equations. Using the symmetry group $G$, we analyze and reduce the contraction equations in each of these three orbits. The reduction of the system $\sS_\Ga$ is performed by equivalent rewriting and by linear operations on the equations, yielding an equivalent system with the same set of solutions. For convenient formulation of the resulting system, it is advantageous to extend the action \eqref{subst}, by acting  on all variables simultaneously, to any equation containing contraction parameters. Thus, we obtain the following result. 

\begin{tvr}
The system of contraction equations $\sS_\Ga$ is equivalent to the system generated by action \eqref{subst} on the equations
\begin{align}
\ev^\bul_{(101)(100)}\ev_{(111)(010)} &=
\ev^\bul_{(111)(110)}\ev_{(010)(100)},\label{eq1}\\
\ev^\ci_{(001)(010)}\ev_{(011)(100)} &=
\ev^\ci_{(001)(110)}\ev_{(010)(100)},\label{eq2}\\
2\ev^\ci_{(001)(100)}\ev^\bul_{(011)(010)} &=
\ev_{(111)(010)}\ev_{(011)(100)}
+\ev_{(011)(110)}\ev_{(010)(100)}.\label{eq3}
\end{align}
The equations \eqref{eq1} and \eqref{eq2} generate $32$ two--term equations and \eqref{eq3} generates $12$ three--term equations.
\end{tvr}
\begin{proof}
We write the contraction equation $e_{(0,1,0)(1,0,0)(1,1,1)}$. Acting by $G$ on this contraction
equation we generate 4 equations. Using the commutation relations \eqref{comrel} we obtain
\begin{align*}
\ev^\bul_{(011)(010)}\ev_{(111)(100)}[e_7,[e_8,e_3]] & +
\ev^\bul_{(101)(100)}\ev_{(111)(010)}[e_8,[e_3,e_7]] + \\
& + \ev^\bul_{(111)(110)}\ev_{(010)(100)}[e_3,[e_7,e_8]] = 0,
\end{align*}
$$
\ev^\bul_{(011)(010)}\ev_{(111)(100)}(-2e_2) +
\ev^\bul_{(101)(100)}\ev_{(111)(010)}(2e_1+2e_2) +
\ev^\bul_{(111)(110)}\ev_{(010)(100)}(-2e_1) = 0.
$$
Since $e_1$ and $e_2$ are linearly independent vectors, we obtain
the following two--term equations
\begin{equation*}
\underbrace{\ev^\bul_{(011)(010)}\ev_{(111)(100)}}_a =
\underbrace{\ev^\bul_{(101)(100)}\ev_{(111)(010)}}_b =
\underbrace{\ev^\bul_{(111)(110)}\ev_{(010)(100)}}_c.
\end{equation*}
This comprises two independent equalities $a=b$, $b=c$ and one
dependent $a=c$. Denoting the set of unordered pairs from $\I$ by  $\I^2_u$ and considering the action of $G$ on $\I^2_u$, one
can see that the indices of the terms $a,b,c$ lie in the same
12--point orbit. Moreover, the matrix $X=\left(
\begin{smallmatrix}
0 & 1 & 0\\
1 & 0 & 0\\
0 & 0 & 1\\
\end{smallmatrix} \right)$ transforms the equation $b=c$ into
the equation $a=c$ and the matrix $Y=\left(
\begin{smallmatrix}
0 & 1 & 0\\
1 & 1 & 1\\
0 & 0 & 1\\
\end{smallmatrix} \right)$ transforms $b=c$ into
$a=b$. Thus, the whole orbit of \eqref{eq1} is generated from the equation
$b=c$ by action of $G$. Since the stability subgroup
$H=\{1,X,Y,Y^2,XY,XY^2\}$ of the point $((0,1,0)(1,0,0)(1,1,1))$
provides all six permutations of the terms $a,b,c$, each of four
left cosets of $G$ --- with respect to the group $H$
--- generates six linearly dependent equations. Among these six
equations only two are linearly independent, therefore, we have $8$
linearly independent equations.

The representative point $((0,1,0)(1,0,0)(0,0,1))$ contains the
index of the 2--dimensional grading subspace $L_{001}$. Therefore,
it leads to two equations
$$
\ev_{(101)(010)}\ev^\ci_{(001)(100)}[e_7,[e_8,e_i]]  +
\ev_{(011)(100)}\ev^\ci_{(001)(010)}[e_8,[e_i,e_7]] +
\ev^\ci_{(001)(110)}\ev_{(010)(100)}[e_i,[e_7,e_8]] = 0, \\
$$
where $i=1,2$. Using the commutation relations \eqref{comrel} we have
\begin{eqnarray*}
(-\ev_{(101)(010)}\ev^\ci_{(001)(100)}
-\ev_{(011)(100)}\ev^\ci_{(001)(010)} +
2\ev^\ci_{(001)(110)}\ev_{(010)(100)})e_3 = 0, \\
(-\ev_{(101)(010)}\ev^\ci_{(001)(100)}+
2\ev_{(011)(100)}\ev^\ci_{(001)(010)}
-\ev^\ci_{(001)(110)}\ev_{(010)(100)})e_3 = 0.
\end{eqnarray*}
By summing and subtracting these equations we obtain new two--term
equations
\begin{equation*}
\underbrace{\ev^\ci_{(001)(100)}\ev_{(101)(010)}}_a =
\underbrace{\ev^\ci_{(001)(010)}\ev_{(011)(100)}}_b =
\underbrace{\ev^\ci_{(001)(110)}\ev_{(010)(100)}}_c.
\end{equation*}
Considering the action of $G$, the matrix $X=\left(
\begin{smallmatrix}
0 & 1 & 0\\
1 & 0 & 0\\
0 & 0 & 1\\
\end{smallmatrix} \right)$ transforms the term $a$ into the term
$b$ while the term $c$ is unchanged. In fact, the index of the term
$c$, i.e. [(001)(110)][(010)(100)], lies in 12--point orbit in
$\I^2_u$ while the indices of $a,b$ belong to the one 24--point
orbit. The stability subgroup of $((0,1,0)(1,0,0)(0,0,1))$ is a
cyclic group $\{1,X\}$ and, therefore, the whole orbit of \eqref{eq1}
consists of 24 linearly independent equations generated from $b=c$
by action of $G$.

The last representative point $((0,1,0)(1,0,0)(0,1,1))$ leads to
three--term equation
\begin{align}\nonumber
\ev_{(111)(010)}\ev_{(011)(100)}[e_7,[e_8,e_5]] +
\ev^\ci_{(001)(100)}\ev^\bul_{(011)(010)}[e_8,[e_5,e_7]] &+
\\ \nonumber
+\ev_{(011)(110)}\ev_{(010)(100)}[e_5,[e_7,e_8]] &= 0, \\
\label{gel_eq3}
(2\underbrace{\ev^\ci_{(001)(100)}\ev^\bul_{(011)(010)}}_a-\underbrace{\ev_{(111)(010)}\ev_{(011)(100)}}_b
-\underbrace{\ev_{(011)(110)}\ev_{(010)(100)}}_c)e_4 &= 0.
\end{align}
The indices of $b,c$ belong to the same 24--point orbit, while the
index of $a$ belongs to 12--point orbit in $\I^2_u$. The matrix
$Z=\left(
\begin{smallmatrix}
1 & 0 & 0\\
0 & 1 & 1\\
0 & 0 & 1\\
\end{smallmatrix} \right)\in G$ transforms $b$ into $c$
while $a$ is preserved. Since the stability subgroup of
$((0,1,0)(1,0,0)(0,1,1))$ is $\{1,Z\}$, we get 12 linearly
independent three--term equations by the action of $G$ on
\eqref{gel_eq3}.
\end{proof}


\subsection{Finding the solution of $\sS_\Ga$}\

An efficient algorithm, developed specially for solving contraction systems, was formulated in \cite{HN1}. It relies on the notion of the equivalence of solutions \eqref{sequiv} -- the idea is not to evaluate all solutions but eliminate equivalent solutions during the solving process. Using this algorithm, we obtain the following explicit list of solutions.    
\begin{tvr}\label{solving}
For any solution of the contraction system $\ep \in\sC_\Ga$ there exist
$a,\,b,\,c,\,d,\,e,\,f\in\bC$ such that $\ep$ is equivalent to some of the following solutions
\begin{alignat*}{4}
 \ev^0_1 &=\left(\begin{smallmatrix}
0 &  1 &  1 &  a &  a &  1 &  a\\
1 &  0 &  b &  c &  cb &  c &  b\\
1 &  b &  0 &  1 &  b &  1 &  b\\
a &  c &  1 &  0 &  ca &  c &  a\\
a &  cb &  b &  ca &  0 &  c & ba\\
1 &  c &  1 &  c &  c &  0 &  1\\
a &  b &  b &  a &  ba &  1 &  0\\
\end{smallmatrix}\right),\,& 
\ev^1_1 &=\left(\begin{smallmatrix}
0 &  0 &  0 &  0 &  0 &  0 &  1\\
0 &  0 &  0 &  0 &  1 &  1 &  1\\
0 & 0 &  0 &  0 &  0 &  0 &  1\\
0 &  0 &  0 &  0 &  1 &  1 &  1\\
0 &  1 &  0 &  1 &  0 &  0 &  1\\
0 &  1 &  0 &  1 &  0 &  0 &  1\\
1 &  1 &  1 &  1 &  1 &  1 &  0\\
\end{smallmatrix}\right),\, &
\ev^1_2 &= \left(\begin{smallmatrix}
0 & 0 &  0 &  0 &  0 &  a &  0\\
0 &  0 &  0 &  0 &  1 &  1 &  1\\
0 &  0 &  0 &  0 &  a &  a &  1\\
0 &  0 &  0 &  0 &  0 &  1 &  0\\
0 &  1 &  a &  0 &  0 &  a &  0\\
a &  1 &  a &  1 &  a &  0 &  1\\
0 &  1 &  1 &  0 &  0 &  1 &  0\\
\end{smallmatrix}\right),\, &
\ev^1_3 &= \left(\begin{smallmatrix}
0 & a &  0 &  0 &  0 &  0 &  0\\ a &  0 &  a &  a &  1 &  1 &  1\\
0 &  a &  0 &  a &  0 &  0 &  1\\ 0 & a &  a &  0 &  0 &  1 &  0\\
0 &  1 &  0 &  0 &  0 &  0 &  0\\  0 &  1 &  0 &  1 &  0 &  0 &  1\\
0 &  1 &  1 &  0 &  0 &  1 &  0\\
\end{smallmatrix}\right), \\
\ev^2_1 &=\left(\begin{smallmatrix}
0 &  0 &  0 &  0 &  0 &  1 &  0\\ 0 &  0 &  0 &  0 &  0 &  0 &  a\\
0 &  0 &  0 &  0 &  1 &  b &  \frac{1}{2}ab+\frac{1}{2}\\
0 &  0 &  0 &  0 &  0 &  0 &  0\\
0 &  0 &  1 &  0 &  0 &  0 & 0\\
1 &  0 &  b &  0 &  0 &  0 &  1\\
0 &  a &  \frac{1}{2}ab+\frac{1}{2} &  0 &  0 &  1 &  0\\
\end{smallmatrix}\right),\,&
\ev^2_2 &=\left(\begin{smallmatrix}
0 &  0 &  0 &  0 &  0 &  1 &  0\\  0 &  0 &  0 &  0 &  0 &  -1 & 1\\
0 &  0 &  0 &  0 &  1 &  1 &  0\\  0 &  0 &  0 &  0 &  0 &  0 &  0\\
0 &  0 &  1 &  0 &  0 & 1 &  0\\  1 &  -1 &  1 &  0 &  1 &  0 &  -1\\
0 &  1 &  0 &  0 &  0 &  -1 &  0\\
\end{smallmatrix}\right),\,&
\ev^3_1 &=\left(\begin{smallmatrix}
0 &  0 &  0 &  0 &  0 &  1 &  0\\  0 &  0 &  0 &  0 &  0 &  a &  0\\
0 & 0 &  0 &  0 &  0 &  b &  0\\  0 &  0 &  0 &  0 &  0 &  c &  0\\
0 &  0 &  0 &  0 &  0 &  d &  0\\ 1 &  a &  b &  c &  d &  0 &  1\\
0 &  0 &  0 &  0 &  0 &  1 &  0\\
\end{smallmatrix}\right),\, &
\ev^3_2 &=\left(\begin{smallmatrix}
0 & a &  0 &  0 &  0 &  1 &  b\\
a &  0 &  0 &  0 &  0 &  c &  d\\  0 &  0 &  0 &  0 &  0 &  0 &  0\\
0 &  0 &  0 &  0 &  0 &  0 &  0\\  0 &  0 &  0 &  0 &  0 &  0 &  0\\
1 &  c &  0 &  0 &  0 &  0 &  1\\  b &  d &  0 &  0 &  0 &  1 &  0\\
\end{smallmatrix}\right),\\
\ev^4_1 &=\left(\begin{smallmatrix}
0 &  0 &  0 &  0 &  0 &  0 &  0\\ 0 &  0 &  -1 &  0 &  0 &  0 &  1\\
0 &  -1 & 0 &  1 &  -a &  a &  b\\  0 &  0 &  1 &  0 &  0 &  0 &  c\\
0 &  0 &  -a &  0 &  0 &  0 &  c\\ 0 &  0 &  a &  0 &  0 &  0 &  1\\
0 &  1 &  b &  c &  c &  1 &  0\\
\end{smallmatrix}\right),\,& 
\ev^4_2 &=\left(\begin{smallmatrix}
0 &  0 &  0 &  0 &  0 &  0 &  0\\  0 &  0 &  1 &  1 &  0 &  1 &  1\\
0 &  1 &  0 & -1 &  -a &  a &  0\\  0 &  1 &  -1 &  0 &  -b &  0 &  b\\
0 &  0 &  -a &  -b &  0 &  a &  b\\ 0 &  1 &  a &  0 &  a &  0 &  1\\
0 &  1 &  0 &  b &  b &  1 &  0\\
\end{smallmatrix}\right)\,&
\ev^5_1 &=\left(\begin{smallmatrix}
0 &  0 &  0 &  0 &  0 &  0 &  0\\  0 &  0 &  0 &  0 &  0 &  0 &  0\\
0 & 0 &  0 &  a &  0 &  b &  c\\  0 &  0 &  a &  0 &  0 &  d &  e\\
0 &  0 &  0 &  0 &  0 &  0 &  0\\ 0 &  0 &  b &  d &  0 &  0 &  1\\
0 &  0 &  c &  e &  0 &  1 &  0\\
\end{smallmatrix}\right),\,&
\ev^5_2 &=\left(\begin{smallmatrix}
0 & 0 &  0 &  0 &  0 &  0 &  0\\  0 &  0 &  0 &  0 &  0 &  0 &  a\\
0 &  0 &  0 &  0 &  0 &  0 &  b\\ 0 &  0 &  0 &  0 &  0 &  0 &  c\\
0 &  0 &  0 &  0 &  0 &  0 &  d\\ 0 &  0 &  0 &  0 &  0 &  0 &  1\\
0 &  a &  b &  c &  d &  1 &  0\\
\end{smallmatrix}\right),
\end{alignat*}
\begin{alignat*}{4}
\ev^5_3 &=\left(\begin{smallmatrix}
0 & 0 &  0 &  0 &  0 &  0 &  0\\0 &  0 &  0 &  0 &  0 &  a &  b\\
0 &  0 &  0 &  0 &  0 &  0 &  0\\ 0 & 0 &  0 &  0 &  0 &  0 &  0\\
0 &  0 &  0 &  0 &  0 &  0 &  0\\  0 &  a &  0 &  0 &  0 &  0 &  1\\
0 & b &  0 &  0 &  0 &  1 &  0\\
\end{smallmatrix}\right),\, &
\ev^5_4 &=\left(\begin{smallmatrix}
0 &  0 &  0 &  0 &  0 &  0 &  0\\ 0 &  0 &  0 &  0 &  0 &  a &  0\\
0 &  0 &  0 &  0 &  0 &  b &  0\\ 0 &  0 &  0 &  0 &  0 &  c & 0\\
0 &  0 & 0 &  0 &  0 &  d &  0\\  0 &  a &  b &  c &  d &  0 &  1\\
0 &  0 &  0 &  0 &  0 &  1 &  0\\
\end{smallmatrix}\right),\, &
\ev^5_5 &=\left(\begin{smallmatrix}
0 &  0 &  0 &  0 &  0 &  0 &  0\\  0 &  0 &  0 &  0 &  0 &  0 & 0\\
0 &  0 &  0 &  0 &  0 &  0 &  0\\  0 &  0 &  0 &  0 &  0 &  0 &  0\\
0 &  0 &  0 &  0 &  0 & a &  b\\  0 &  0 &  0 &  0 &  a &  0 &  1\\
0 &  0 &  0 &  0 &  b &  1 &  0\\
\end{smallmatrix}\right),\, &
\ev^5_6 &=\left(\begin{smallmatrix}
0 &  a &  0 &  0 &  0 &  0 &  0\\ a &  0 &  0 &  0 &  0 &  0 & 0\\
0 &  0 &  0 &  b &  0 &  0 &  0\\  0 &  0 &  b &  0 &  0 &  0 &  0\\
0 &  0 &  0 &  0 & 0 &  0 &  0\\  0 &  0 &  0 &  0 &  0 &  0 &  1\\
0 &  0 &  0 &  0 &  0 &  1 &  0\\
\end{smallmatrix}\right),\\
\ev^6_1 &=\left(\begin{smallmatrix}
0 &  a &  b &  c &  d &  e &  f\\  a &  0 &  0 &  0 &  0 &  0 &  0\\
b & 0 &  0 &  0 &  0 &  0 &  0\\  c &  0 &  0 &  0 &  0 &  0 &  0\\
d &  0 &  0 &  0 &  0 &  0 &  0\\ e &  0 &  0 &  0 &  0 &  0 &  0\\
f &  0 &  0 &  0 &  0 &  0 &  0\\
\end{smallmatrix}\right),\, &
\ev^6_2 &=\left(\begin{smallmatrix}
0 & 0 &  a &  0 &  0 &  0 &  b\\  0 &  0 &  0 &  0 &  0 &  0 &  0\\
a &  0 &  0 &  0 &  0 &  0 &  c\\ 0 &  0 &  0 &  0 &  0 &  0 &  0\\
0 &  0 &  0 &  0 &  0 &  0 &  0\\ 0 &  0 &  0 &  0 &  0 &  0 &  0\\
b &  0 &  c &  0 &  0 &  0 &  0\\
\end{smallmatrix}\right), \, &
\ev^6_3 &=\left(\begin{smallmatrix}
0 & 0 &  0 &  a &  0 &  b &  0\\  0 & 0 &  0 &  0 &  0 &  0 &  0\\
0 &  0 &  0 &  0 &  0 &  0 &  0\\ a & 0 &  0 &  0 &  0 &  c &  0\\
0 &  0 &  0 &  0 &  0 &  0 &  0\\  b &  0 &  0 &  c &  0 &  0 &  0\\
0 & 0 &  0 &  0 &  0 &  0 &  0\\
\end{smallmatrix}\right), \, &
\ev^6_4 &=\left(\begin{smallmatrix}
0 &  a &  0 &  0 &  b &  0 &  0\\ a &  0 &  0 &  0 &  c &  0 &  0\\
0 &  0 &  0 &  0 &  0 &  0 &  0\\ 0 &  0 &  0 &  0 &  0 &  0 & 0\\
b &  c & 0 &  0 &  0 &  0 &  0\\  0 &  0 &  0 &  0 &  0 &  0 &  0\\
0 &  0 &  0 &  0 &  0 &  0 &  0\\
\end{smallmatrix}\right), 
\end{alignat*}
\begin{equation*}
\ev^6_5 =\left(\begin{smallmatrix}
0 &  0 &  0 &  0 &  0 &  0 &  0\\  0 &  0 &  0 &  0 &  a &  0 & 0\\
0 &  0 &  0 &  0 &  0 &  0 &  b\\  0 &  0 &  0 &  0 &  0 &  c &  0\\
0 &  a &  0 &  0 & 0 &  0 &  0\\  0 &  0 &  0 &  c &  0 &  0 &  0\\
0 &  0 &  b &  0 &  0 &  0 &  0\\
\end{smallmatrix}\right).
\end{equation*}
Moreover, any two solutions $\ep^m_k$, $\ep^n_l$ with $m\neq n$ are not equivalent.   
\end{tvr}
\begin{proof}
Repeatedly using Theorem 7 from \cite{HN1} we solve the system $\sS_\Ga$ in seven steps:
\begin{enumerate}
\item The explicit solution of the system $\sS_\Ga$ under the
assumptions $\ev_{(010)(100)}\neq 0$ and $\ev^\ci_{(001)(101)}\neq
0$ is given as a matrix which is strongly equivalent to $\ev^0_1$.
\item In order to eliminate solutions equivalent to those obtained in the previous step, the system generated from the equation $\ev_{(010)(100)}\ev^\ci_{(001)(101)}=0$ has to be satisfied. This non--equivalence system is denoted by $\mS^0$.
Solutions of $\sS_\Ga$ and $\mS^0$ under the assumptions
$\ev_{(010)(100)}\neq 0$ and $\ev^\bul_{(111)(110)}\neq 0$ are
described by $3$ parametric matrices which are strongly equivalent to the matrices $\ev^1_1,\, \ev^1_2$ and $\ev^1_3$.
\item The solutions of $\sS_\Ga\cup \mS^0$ together with the non--equivalence system $\mS^1$ generated by $$\ev_{(010)(100)}\ev^\bul_{(111)(110)}=0$$ under the assumptions
$\ev_{(010)(100)}\neq 0, \ev^\ci_{(001)(010)}\neq 0$, and
$\ev_{(101)(110)}\neq 0$ are given by two parametric matrices
$\ev^2_1$ and $\ev^2_2$. 
\item The solutions of $\sS_\Ga\cup \mS^0 \cup\mS^1 $ together with the non--equivalence system $\mS^2$ generated by
$$\ev_{(010)(100)}\ev^\ci_{(001)(010)}\ev_{(101)(110)}=0$$ under the assumptions $\ev_{(010)(100)}\neq 0,
\ev^\ci_{(001)(010)}\neq 0$ are strongly equivalent to the solutions
$\ev^3_1$ and $\ev^3_2$.
\item The solutions of $\sS_\Ga\cup \mS^0 \cup\mS^1 \cup\mS^2$ together with the non--equivalence system $\mS^3$ generated by
$$\ev_{(010)(100)}\ev^\ci_{(001)(010)}=0$$ under the assumptions $\ev_{(010)(100)}\neq 0$, $\ev_{(111)(101)}\neq 0$ are strongly equivalent to the solutions
$\ev^4_1$ and $\ev^4_2$.
\item The solutions of $\sS_\Ga\cup \mS^0 \cup\mS^1 \cup\mS^2\cup \mS^3$ together with the non--equivalence system $\mS^4$ generated by
$$\ev_{(010)(100)}\ev_{(111)(101)}=0$$ under the assumption $\ev_{(010)(100)}\neq 0$ are strongly equivalent to the solutions
$\ev^5_1,\,\ev^5_2,\,\ev^5_3,\,\ev^5_4,\,\ev^5_5$ and $\ev^5_6$.
\item Finally, the non--equivalence system
$\mS^5$ generated by $\ev_{(010)(100)}=0$ enforces zero
value of all contraction parameters in the orbit --- these variables are not marked by '$+$' or '$-$' --- and ensures fulfilment
of all previous non--equivalence systems. Due to these $12$ zero
contraction parameters, all two--term equations are satisfied and
three--term equations are reduced to the system generated by
$\ev^\ci_{(001)(100)}\ev^\bul_{(011)(010)}=0$. Corresponding solutions are strongly equivalent to
$\ev^6_1,\,\ev^6_2,\,\ev^6_3,\,\ev^6_4$ and $\ev^6_5$.
\end{enumerate}
\end{proof}

Even though any two solutions $\ep^m_k$, $\ep^n_l$ with $m\neq n$ are not equivalent in Proposition \ref{solving}, the list of solutions still contains equivalent solutions if $m=n$. Since equivalent solutions have the same number of zeros $\nu(\ev)=|\set{(i\,j) \in \I}{\ev_{ij}= 0}|$,
we discuss when elements of contraction matrices vanish and divide
the solutions according to $\nu(\ev)$. Besides all possible combinations of zero and non--zero
parameters, the case $a=-1/b$ has to be also considered for the solution
$\ev^2_1$. After that, we collect all solutions with the same
support $S(\ev) = \set{(i\,j)\in \I}{\ev_{ij} \neq 0}$. For example, solutions $\ev^5_4$ with $a=b=c=0$ and
$\ev^5_5$ with $b=0$ have the same support $S(\ev^5_4)=S(\ev^5_5)=\{{(110)(010), (010)(100)}  \}$ but they also represent
the same solution --- only with different notation of the parameters. 

The notion of a projection $\hat\ev$ of a solution $\ev$ defined via $\hat\ev_{ij}=\operatorname{sgn} |\ev_{ij}|$ is useful for sorting the solutions of contraction equations. Since there is only one solution with the given
support $S(\ev)$ and thus $\hat\ev^1 \sim \hat\ev^2$ would imply $\ev^1 \sim \ev^2$, we eliminate those contraction matrices which have equivalent projections. Let us note that the solution without zeros, i.e. $\ev^0_1$, where
$a,b,c,d\neq0$, is strongly equivalent to the trivial solution
$\ev^{0,1}$ which has all relevant contraction parameters equal to
1. Thus, any contraction matrix without zeros has the form of the
normalization matrix \eqref{normat}. Thus, by purging the list of solutions for overlaps and equivalencies, we obtain $89$ normalized representatives of equivalence classes of solutions of $\sS_\Ga$.

It remains to answer the question of equivalence of the system of contraction equations \eqref{general} with simplified two--term systems~\eqref{extend} and~\eqref{ropa}. It can be verified directly that both systems~\eqref{extend} and~\eqref{ropa} reduce the number of solutions of ~\eqref{general}:
\begin{enumerate}[(i)]
\item among $89$ contraction matrices, there are 55 solutions which fulfill the system~\eqref{extend}. On the contrary, there are 6 parametric solutions which fulfill this system only if all their parameters are equal to $1$ and $28$ non--parametric solutions which do not solve this system of two--term equations,

\item there are $74$ contraction matrices which satisfy the system~\eqref{ropa}. On the contrary, there is one parametric contraction matrix which fulfils this system only if its parameter is equal to $1$ and $14$ contraction matrices which do not fulfill this system of two--term equations at all.
\end{enumerate} 
The final list of all $89$ normalized representatives of equivalence
classes of solutions, with the solutions which do not satisfy the simplified systems~\eqref{extend} and~\eqref{ropa} marked, is located in Appendix A.1.

\subsection{Continuous and discrete graded contractions}\

A solution $\ep \in\sC_\Ga$ is called continuous if there exist continuous functions $$a_i:(0,1]
\rightarrow \bC\setminus\{0\},\ i\in I$$ such that for all relevant contraction parameters 
$$\ev_{ij} = \lim\limits_{t\rightarrow 0} \frac{a_i(t)a_j(t)}{a_{i+j}(t)}$$
holds; otherwise it is called discrete. If
the functions $a_i$ are of the form $a_i(t) = t^{n_i}$, $n_i\in \bZ$ then a continuous graded contraction becomes a generalized In\"on\"u--Wigner contraction \cite{WW2}. 

An efficient tool for distinguishing between continuous and
discrete graded contractions was developed in \cite{WW2}. So
called higher--order identities allow us to identify discrete
graded contractions. Let us consider an equation where on both
sides stand products of $r$ relevant contraction parameters. If
this equation holds for all normalization matrices
$\alpha_{ij}(t)=\frac{a_i(t)a_j(t)}{a_{i+j}(t)}$, it will also hold for their limit,
i.e. for any continuous solution. Thus, any
equation of the type
\begin{equation*}
\ev_{i_1}\ev_{i_2}\dots\ev_{i_r}=\ev_{j_1}\ev_{j_2}\dots\ev_{j_r},
\end{equation*}
where $r\in \bN$ and $\{i_1,i_2,\ldots,i_r\}$, $\{j_1,j_2,\ldots,j_r\}$
are disjoint sets of relevant pairs of
grading indices, is called higher--order identity of order $r$, if it
holds for all normalization matrices, but is violated by some contraction matrix from $C_\Gamma$.

Higher--order identities can be deduced from the identities
which hold for the normalization matrix~\eqref{normat}. For example the equation
\begin{equation}\label{gell_hoipr}
\alpha_{(001)(100)}\alpha_{(011)(010)} =
\frac{a_{(001)}a_{(100)}}{a_{(101)}}\frac{a_{(011)}a_{(010)}}{a_{(001)}}
=
\frac{a_{(011)}a_{(110)}}{a_{(101)}}\frac{a_{(010)}a_{(100)}}{a_{(110)}}
= \alpha_{(011)(110)}\alpha_{(010)(100)}
\end{equation}
is evidently satisfied for any normalization matrix. However,
considering the contraction matrix $\ev^4_2$ we get $0=-b$ and thus
\eqref{gell_hoipr} is violated for any $b\neq 0$. Therefore, the
equation \eqref{gell_hoipr} represents second order identity and the solution $\ev^4_2$ with  $b\neq 0$ is discrete. Applying the symmetry group $G$ to \eqref{gell_hoipr}, we can
write the 24-point orbit of second order identities generated from the equation
\begin{equation*}
\ev^\ci_{(001)(100)}\ev^\bul_{(011)(010)} =
\ev_{(011)(110)}\ev_{(010)(100)}.
\end{equation*}

In a similar way, all $57$ second order identities are found. These
identities are divided into 5 orbits and their representatives and the
number of the resulting identities under the action of $G$ are
written in Table~\ref{tab:hoig}. For each solution of the system
$\sS_\Ga$, the set of second order identities allows us to distinguish whether the solution is continuous or discrete.
Any discrete contraction violates at least one of the identities
listed in Table \ref{tab:hoig}. The remaining solutions are explicitly found as limits of
$\alpha_{ij}(t)=t^{n_i+n_j-n_{i+j}}$, i.e. they correspond to generalized In\"on\"u--Wigner contractions. Among $89$ solutions in Appendix A.1 there are $50$ continuous
and $36$ discrete ones. The remaining 3 solutions are
continuous only for a special value of their parameters, otherwise
they are discrete. 

\begin{table}
\begin{center}
\renewcommand{\arraystretch}{1.50}
\begin{tabular}[b]{|c|c|}
 \hline
Representative equation & Number of equations \\
\hline\hline $ \ev^\ci_{(001)(100)}\ev^\bul_{(011)(010)} = \ev_{(011)(110)}\ev_{(010)(100)} $ & $24$\\
\hline $ \ev^\ci_{(001)(100)}\ev^\bul_{(101)(100)} = \ev_{(110)(100)}\ev_{(010)(100)}$ & $12$\\
\hline $ \ev_{(111)(010)}\ev_{(011)(100)} = \ev_{(011)(110)}\ev_{(010)(100)}$ & $12$\\
\hline $ \ev_{(111)(100)}\ev_{(011)(100)} = \ev_{(110)(100)}\ev_{(010)(100)}$ & $6$\\
\hline $ \ev^\ci_{(001)(101)}\ev^\ci_{(001)(100)} =\ev^\ci_{(001)(111)}\ev^\ci_{(001)(110)}$ & $3$\\
\hline
\end{tabular}
\end{center}
\smallskip
\caption{Orbits of the second order identities for the Gell-Mann
grading of $\slc$.}\label{tab:hoig}
\end{table}

\section{The contracted Lie algebras}

\subsection{Identification of the contracted algebras}\

All $89$ solutions of the system of contraction equations $\sS_\Ga$ for
the Gell--Mann graded $\spl(3,\bC)$ are divided into 14 groups
according to the number of zeros $\nu$ among the 21 relevant
contraction parameters. The numbers of contraction matrices in these
groups are summarized in the following table:
$$
\begin{tabular}{l||c|c|c|c|c|c|c|c|c|c|c|c|c|c}
\parbox[l][18pt][c]{0pt}{} Number of zeros $\nu$ & 0 & 6 & 9 & 11 & 12 & 13 & 14 & 15 & 16 & 17 & 18 & 19 & 20 & 21 \\
\hline \parbox[l][18pt][c]{0pt}{} Number of solutions & 1 & 2 & 1 & 3 & 2 & 1 & 2 & 9 & 12 & 18 & 23 & 11 & 3 & 1 \\
\end{tabular}
$$
Contraction matrices are denoted $\varepsilon^{\nu,i}$, where the
second index $i$ is numbering solutions with the same number of
zeros $\nu$. The contracted Lie algebra given by solution
$\varepsilon^{\nu,i}$ is denoted $\mG_{\nu,i}$.

There are two trivial solutions: $\varepsilon^{21,1}$
(with 21 zeros) corresponding to the 8--dimensional abelian Lie and
$\varepsilon^{0,1}$ (without zeros) corresponding to the initial Lie
algebra $\spl(3,\bC)$. Among the remaining 87 nontrivial solutions,
8 solutions depend on one nonzero complex parameter $a$ and two
depend on two nonzero complex parameters $a$, $b$. The corresponding
parametric families of Lie algebras --- the parametric Lie algebras
--- are denoted by $\mG_{\nu,i}(a)$, $\mG_{\nu,i}(a,b)$. Each of
these parametric Lie algebras will be counted as one algebra.

In order to classify our results we use an extension of the identification procedure from \cite{HN1}.
We have to identify 87/10 Lie algebras --- the number following the slash refers to the number of parametric families among all 87 algebras. Our extended algorithm for identification, applied to the Gell-Mann graded contractions, consists of the following steps (for more detailed calculation algorithm of steps 1,2 and 5 see \cite{ide}):
\leftmargini 20pt
\begin{enumerate}
\item \textit{Splitting of the maximal central component} \\
Whenever the complement of the derived algebra $D(\mL)=[\mL,\mL]$ into the center $C(\mL)$ of the Lie algebra $\mL$ is nonzero, $\mL$ can be decomposed into the direct sum $\mL=\mL' \oplus k\mA_1$, where $k=\dim C(\mL)/(C(\mL)\cap D(\mL))$, $k\mA_1$ is an abelian algebra of dimension $k$ and non--abelian part $\mL'$ fulfills $C(\mL')\subseteq D(\mL')$. The separation of maximal central component is possible in 66/8 cases of our contracted Lie algebras. Further, we proceed with non--abelian parts only.

\item \textit{Decomposition into a direct sum of indecomposable ideals} \\
A complex Lie algebra $\mL$ is decomposable if and only if there exist an idempotent $0\neq E=E^2 \neq 1$ in the centralizer $C_R:=\set{x\in R}{[x,y]=0,\ \forall y\in \ad(\mL)}$ of the adjoint representation of $\mL$ in the ring $R=\bC^{n,n}$ of all $n\times n$ complex matrices. If so then $\mL = \mL_0 \oplus \mL_1$ where $\mL_0$ and $\mL_1$ are eigen-subspaces of the idempotent $E$ corresponding to the eigenvalues $0,1$. There are only 7/1 decomposable Lie algebras among our results. All of them decompose into the sum of two indecomposable ideals. From now on, we proceed with indecomposable Lie algebras only. Thus, in further steps we deal with 94/10 indecomposable Lie algebras. These algebras are divided according to their dimensions
as follows:
\begin{center}
\begin{tabular}{l||c|c|c|c|c}
\parbox[l][18pt][c]{0pt}{} Dimension         &  3   &    5   &  6  &   7   &  8\\
\hline \parbox[l][18pt][c]{0pt}{} Number of algebras  & 22/1 &  12/2 &  10 &  29/5 & 21/2\\
\end{tabular}
\end{center}

\item \textit{Series of ideals} \\
We calculate the derived series $ D^0(\mL) = \mL, \ D^{k+1}(\mL) = [D^k(\mL),D^k(\mL)]$, the lower central series $\mL^1 = \mL,\ \mL^{k+1} = [\mL^k,\mL]$ and the upper central series $C^1(\mL) = C(\mL), \ C^{k+1}(\mL)/C^k(\mL) = C(\mL/C^k(\mL))$ for each Lie algebra. Dimensions of ideals in theses series are invariants of Lie algebra $\mL$, therefore we divide all investigated Lie algebras into the classes according to these invariants. The results are 54/1 nilpotent, 33/9 solvable (non--nilpotent) and 7 non--solvable Lie algebras.

\item \textit{$(\alpha,\beta,\gamma)$--derivations} \\
Let $\alpha,\beta,\gamma$ be complex numbers. The dimension of the space of $(\alpha,\beta,\gamma)$--derivations of $\mL$
\begin{equation*}
\der_{(\alpha,\beta,\gamma)}\mL = \set{A\in \gl(\mL)}{\alpha A[x,y] = \beta[Ax,y] + \gamma[x,Ay], \ \forall x,y\in\mL}
\end{equation*} is an invariant of Lie algebra $\mL$ \cite{HN5}. As a special case, the algebra of derivations is also obtained $\der \mL=\der_{(1,1,1)}\mL $. We denote by
$\dim_{(\alpha,\beta,\gamma)}\mL$ the 6--tuple formed by the dimensions of the following spaces
\begin{equation*}
\quad\quad \dim_{(\alpha,\beta,\gamma)}\mL= [\der\mL,\ \der_{(0,1,1)}\mL,\ \der_{(1,1,0)}\mL,\
\der\mL\cap\der_{(0,1,1)}\mL,\
\der_{(1,1,-1)}\mL,\ \der_{(0,1,-1)}\mL].
\end{equation*}
Values of these invariants divide our algebras
into 28 classes of nilpotent, 17 classes of solvable and 4 classes
of non--solvable Lie algebras.

\item \textit{Determination of the radical, the Levi decomposition and the nilradical} \\
There are 7 non--solvable Lie algebras for which we determine radicals according to $R(\mL) = \set{x\in\mL}{\Tr(\ad(x)\ad(y))=0, \, \forall y\in D(\mL)}$ and find their Levi decompositions. The Levi decomposition is nontrivial only in 3 cases, two of them with abelian radical. For solvable Lie algebras we determine their nilradicals $N(\mL)$.

\item \textit{Casimir operators}\\
The elements in the center of the universal enveloping Lie algebra $U(\mL)$ of Lie algebra $\mL$ -- Casimir operators -- can be found as follows \cite{Alonso}. Represent elements of basis
$(e_1,\ldots,e_n)$ in $\mL$ by the vector fields
\begin{equation*}
e_i \rightarrow \widehat x_i = \sum_{j,k=1}^n c_{ij}^k x_k
\frac{\partial}{\partial x_j},
\end{equation*} which act on the space of continuously differentiable functions
$F(x_1,\ldots,x_n)$ of $n$ variables. Find formal invariants i.e. such functions $F$ which solve the following linear system of first-order partial differential equations $\widehat{x_i} F(x_1,\ldots,x_n) =0 $ for all $i=1,\ldots,n$.
The number of functionally
independent solutions of this system
\begin{equation*}
\tau(\mL) = \dim(\mL) - \sup_{x_1,\ldots,x_n\in\bC} \operatorname{rank}
M_\mL,
\end{equation*} where $M_\mL$ is skew--symmetric matrix with entries
$(M_\mL)_{ij} = \sum_{k=1}^n c_{ij}^k x_k$, is an invariant of $\mL$. Casimir operators correspond to the polynomial formal invariants.  This correspondence is
provided by symmetrization: any term
$x_{k_1}\ldots x_{k_p}$ of polynomial $F(x_1,\ldots,x_n)$ in
commuting variables $x_i$ is replaced by symmetric term in
non--commuting basis elements $e_1,\ldots,e_n\in\mL$ as follows
\begin{equation*}
x_{k_1}\ldots x_{k_p} \mapsto \frac{1}{p\, !}\sum_{\sigma\in S_p}
e_{k_{\sigma(1)}}\ldots e_{k_{\sigma(p)}}.
\end{equation*} We determine the number of formal invariants $\tau(\mL)$ for all investigated algebras. Since all independent formal invariants of complex nilpotent Lie algebras can be found in polynomial form, we also determine all Casimir operators for nilpotent Lie algebras.

\item \textit{Isomorphisms}\\
Having two $n$-dimensional Lie algebras $\mL$ and $\wt\mL$ with the same values of
all above listed invariant characteristics, we search for isomorphisms explicitly. Thus, we solve the following system of $n^2(n-1)/2$ quadratic equations 
\begin{equation*}
\sum_{r=1}^{n} c_{ij}^{r}A_{kr} = \sum_{\mu, \nu =1}^{n} \tilde
c_{\mu\nu}^k A_{\mu i}A_{\nu j}, \qquad i=1,\ldots,n-1,\quad j=
i,\ldots,n,\quad k\in\hat{n},
\end{equation*}
where $c_{ij}^k, \tilde
c_{ij}^k$ are structural constants of $\mL, \wt\mL$ and $A_{i,j}$ are components of $n\times n$ complex regular matrix representing the isomorphism.
Solving this system for all Lie algebras in the same class we find: 26 isomorphic algebras in
the classes of nilpotent, 16/4 isomorphic algebras in the classes of
solvable and 3 isomorphic algebras in the classes of non--solvable
Lie algebras. Omitting these algebras we get only one algebra in
each class. Thus, all algebras are now identified up to
ranges of parameters for the parametric algebras. 

\item \textit{Invariant functions}\\
The invariants listed above are not able to distinguish Lie algebras within parametric families. Therefore we use the concept of invariant functions \cite{HN5,HN6}, which enables us to classify one--parametric families of Lie algebras.
The invariant function $\psi_\mL$ which arises from $(\alpha,\beta,\gamma)$--derivations is defined by
\begin{equation*}
\psi_{\mL}(\alpha) = \dim(\der_{(\alpha,1,1)}\mL).
\end{equation*}
We calculate this invariant function for all one--parametric Lie algebras. Together with isomorphisms it allows us to determine the ranges of parameters for 3 one--parametric families. The invariant functions $\psi_\mL$ of these three cases are listed in Appendix A.3.
\end{enumerate}

\subsection{The results of identification}\

There are only 4 mutually non--isomorphic decomposable Lie algebras
among the graded contractions of Gell--Mann graded $\spl(3,\bC)$:
\begin{center}
\renewcommand{\arraystretch}{1.25}
\begin{tabular}{lr}
Non--solvable (discrete contraction) & $\mG_{9,1} \cong \mG'_{18,8} \oplus \mG'_{18,8} \oplus 2\mA_1$ \\
Solvable (discrete contractions) & $\mG_{17,2}(a) \cong \mG_{13,1} \cong \mG'_{19,2} \oplus \mG'_{19,2} \oplus 2\mA_1$\\
 & $\mG_{18,2} \cong \mG'_{19,2} \oplus \mG'_{20,1} \oplus 2\mA_1$ \\
Nilpotent (continuous contractions) & $ \mG_{19,1} \cong \mG_{19,11}
\cong\mG_{17,18} \cong \mG'_{20,1} \oplus \mG'_{20,1} \oplus 2\mA_1$\\
\end{tabular}
\end{center}
where $\mA_1$ stands for one--dimensional abelian Lie algebra.

The list of all isomorphisms among
the indecomposable graded contractions of the Gell--Mann graded $\spl(3,\bC)$ is presented in Table~\ref{tab:iso}.
\begin{table}
\begin{center}
\renewcommand{\arraystretch}{1.25}
\tabcolsep 10pt
\begin{tabular}{|l||l|l|}
\hline
\multicolumn{1}{|c||}{Non--solvable} & \multicolumn{2}{c|}{Nilpotent} \\
\hline  $\mG_{18,21} \cong \mG_{18,8}$ & $\mG_{20,3} \cong \mG_{20,2} \cong \mG_{20,1}$ & $\mG_{19,10} \cong  \mG_{19,4}$   \\
\hline \multicolumn{1}{|c||}{Solvable} & $\mG_{19,9} \cong \mG_{19,7} \cong \mG_{19,5} \cong \mG_{19,3}$ & $\mG_{18,17} \cong \mG_{18,16}$ \\
\hline  $\mG_{19,8} \cong \mG_{19,6} \cong \mG_{19,2}$ & $\mG_{18,20} \cong \mG_{18,19} \cong \mG_{18,7}$  & $\mG_{18,3} \cong \mG_{18,1}$ \\
\hline $\mG_{18,14} \cong\mG_{18,13} \cong \mG_{18,11} \cong \mG_{18,9}$ & $\mG_{18,15} \cong \mG_{18,12} \cong \mG_{18,10} \cong \mG_{18,6}$ & $\mG_{17,17} \cong \mG_{17,13}$\\
\hline $\mG_{17,10}\cong\mG_{17,9} \cong \mG_{17,7} \cong \mG_{17,6}$ & $\mG_{18,22} \cong \mG_{18,4} \cong \mG_{16,12}$ & $\mG_{17,14} \cong \mG_{17,12}$ \\
\hline  $\mG_{17,8}(a) \cong \mG_{17,11}(4a)$ & $\mG_{16,10} \cong \mG_{16,8}$ & $\mG_{16,11} \cong \mG_{16,7}$\\
\hline  $\mG_{16,3}(a) \cong \mG_{16,5}(4a) \cong \mG_{16,4}(4a)$ & & \\
\hline
\end{tabular}
\end{center}
\medskip
\caption{Isomorphisms among the indecomposable $\Ga-$graded contractions.}\label{tab:iso}
\end{table}
Note that isomorphic graded contractions are always of the
same type, i.e. all discrete or all continuous.
Among the resulting non-isomorphic indecomposable algebras, there are $28$
nilpotent, one of them being parametric, 17 solvable with 5 of them parametric and 4 non--solvable. 
These indecomposable Lie algebras are
listed together with their invariant characteristics in Appendix A.2.
All resulting decomposable Lie algebras can be written as their
direct sums.

Including two trivial contractions we have obtained 55
non--isomorphic contracted Lie algebras as the graded contractions
of the Gell--Mann graded Lie algebra $\spl(3,\bC)$. Among them there
are 4 one--parametric and 2 two--parametric families of Lie
algebras. From all these contracted Lie algebras 20 are discrete
contractions, 32 continuous contractions and 3 parametric algebras
represent continuous contractions for a special value of their
parameter; otherwise they are discrete.
Table \ref{tab:gellmann} provides the overview of the numbers of
contracted Lie algebras for the Gell--Mann graded $\spl(3,\bC)$. Lie
algebras are divided there according to the dimension of their
non--abelian parts and their types.

\begin{table}
 \center
\renewcommand{\arraystretch}{1.2}
\begin{tabular}{|c||c|c|c|c|c|c||c|}
\hline Dimension of& \multicolumn{2}{c|}{Solvable} &
\multicolumn{2}{c|}{Nilpotent} & \multicolumn{2}{c||}{Non--solvable} & Total \\
\cline{2-7} non--abelian part & Indec. & Dec. & Indec. & Dec. &
Indec. & Dec. &
\\
\hline
\hline  3 & 1 &  & 1 & & 1 & & 3\\
\hline  4 &  &  &  &  & & & \\
\hline  5 & 2 &  & 2 &  & & & 4\\
\hline  6 & 2 & 2 & 3 & 1 & 1 & 1 & 10\\
\hline  7 & 6 &  & 10 &  & & & 16\\
\hline  8 & 6 &  & 12 &  & 2 & & 20\\
\hline \hline & & & & & & & 53\\
\hline
\end{tabular}
\medskip
\caption{The numbers of the non--trivial graded contractions of the
Gell-Mann graded $\spl(3,\bC)$.}\label{tab:gellmann}
\end{table}

It remains to compare our list of the graded contractions of the Gell-Mann graded $\spl(3,\bC)$ with the previous works which contain results for the Cartan grading \cite{PC4} and the Pauli grading \cite{HN1} of $\spl(3,\bC)$. 
Since there is a common coarsening for the Cartan and the Gell-Mann
gradings, it is expectable that there will be also common
Lie algebras among the corresponding non--trivial
contractions. There is no common coarsening, however, for the Pauli and the  Gell-Mann grading. Therefore, not all common results are the consequence of the existence of a common coarsening.
The contracted Lie algebras $\mG_{19,2},\   \mG_{15,1}(a,b),\ \mG_{11,2}$ and $\mG_{18,6}$ appear as graded contractions of the Cartan grading, the algebras $\mG_{15,2}({\textstyle\frac{-1+\sqrt{3}i}{8},\frac{-1-\sqrt{3}i}{2}}),\ \mG_{19,4},\
\mG_{18,7},\ \mG_{17,13},\ \mG_{18,18},\ \mG_{17,5},\ \mG_{18,5},\  \mG_{15,6}({\textstyle\frac{-1+\sqrt{3}i}{2}})$ and $\mG_{16,2}$ appear as graded contractions
of the Pauli Grading and the algebras $\mG_{18,8},\ \mG_{6,1},\ \mG_{15,1}(1,1),\ \mG_{20,1},\mG_{17,18},\ \mG_{19,3},\  \mG_{18,16},$ $\mG_{16,12},$ $\mG_{18,23},\ \mG_{15,9}$  and $\mG_{15,5}$ appear as graded contractions
of both Cartan and Pauli gradings.

\section{Concluding remarks}

\begin{itemize}
\item The richness of the outcome of the graded contraction procedure and its crucial dependence on the initial grading is illustrated by the fact that $16$ solvable algebras, including two--parametric continuum, which are not among the results neither from Pauli nor Cartan gradings, are obtained. Moreover, $12$ nilpotent including one--parametric continuum and three non-solvable algebras are also new among the graded contractions of $\slc$. 
\vspace{-0pt}
\item  The physically important rigid rotor algebra $[\mathbb{R}^5]$so$(3)$, see \cite{Guise2} and references therein, is a real eight--dimensional algebra generated by five commuting quadrupole moments $Q_\nu,\,\nu=\pm2, \pm 1, 0$ 
and three angular moments $L_i, i=\pm 1, 0$. 
It is found among the results as a real form of the algebra $\mG_{6,2}$ with commutation relations listed in Appendix A.2, where $Q_{-2}=e_1,\, Q_{-1}=e_2, \dots, Q_{2}=e_5$ and $L_{-1}=e_6,\, L_{0}=e_7,\, L_{1}=e_8 $ are identified. 

\item It remains to apply the graded contraction procedure to the last of the four fine gradings of $\slc$, so called $\Ga_d$ grading \cite{HPPT1}. In contrast to Cartan, Pauli and Gell-Mann gradings, this finest
grading $\Ga_d$ is known for Lie algebra $\slc$ only ---
it decomposes $\slc$ into eight one--dimensional grading subspaces determined by the matrices
\begin{equation*}
\begin{array}{llllllll}
\left(\begin{smallmatrix}
        0&0&0\\
        0&1&0\\
        0&0&-1
        \end{smallmatrix}\right), &
\left(\begin{smallmatrix}
        0&1&0\\
        0&0&0\\
        -1&0&0
        \end{smallmatrix}\right), &
\left(\begin{smallmatrix}
        0&0&0\\
        0&0&1\\
        0&0&0
        \end{smallmatrix}\right), &
\left(\begin{smallmatrix}
        0&0&1\\
        1&0&0\\
        0&0&0\\
        \end{smallmatrix}\right),&
\left(\begin{smallmatrix}
        2&0&0\\
        0&-1&0\\
        0&0&-1
        \end{smallmatrix}\right), &
\left(\begin{smallmatrix}
        0&1&0\\
        0&0&0\\
        1&0&0
        \end{smallmatrix}\right), &
\left(\begin{smallmatrix}
        0&0&0\\
        0&0&0\\
        0&1&0
        \end{smallmatrix}\right), &
\left(\begin{smallmatrix}
        0&0&1\\
        -1&0&0\\
        0&0&0
        \end{smallmatrix}\right).
\end{array}
\end{equation*}
Only non-solvable contracted algebras of $\Ga_d$ grading are classified \cite{PN}. Since the symmetry group of $\Ga_d$ has only four elements and a vast number of parametric solutions appears, a complete classification of the outcome of the graded contraction procedure is still beyond reach. 
\item For the study of the four one--parametric families of Lie
algebras, the concept of generalized derivations \cite{HN5} is used and the corresponding invariant functions are calculated and tabulated in Appendix A.3. Even though this concept was extended to so called twisted cocycles in \cite{HN6} and is able to handle two--parametric continua of four--dimensional Lie algebras, the complexity of calculations in dimension eight still prevents reaching a classification of the two two--parametric continua in the present case. Thus, these two two--parametric algebras warrant further study.    
\item Even though the possibility of existence of a three--term contraction equation is brought forward in \cite{HN1}, the solution $\ev^{9,1}$ is the first solution to appear in the literature which would be lost if only two--term equations were considered. A detailed comparison of the results of the two simplifying approaches \eqref{extend} and \eqref{ropa}, which both use two--term equations only, is located in Appendix A.1. 
\item The $33$ non--trivial IW contractions form a significant part among the solutions of the graded contraction procedure and are distinguished among the contraction matrices in Appendix A.1, as well as among the contracted algebras in Appendix A2. Thus, among a few others, this procedure is a very effective tool for detecting and explicit formulation of this physically important kind of relation among Lie algebras.
\end{itemize}
\section*{Acknowledgments}
The authors gratefully acknowledge the support of this work by RVO68407700 and by the Ministry
of Education of Czech Republic from the project MSM6840770039. 

\section*{Appendix A.1: Contraction matrices}

Non--equivalent solutions
of contraction system for the Gell--Mann graded Lie algebra $\slc$ are listed. 
These solutions are divided according to
the number of zeros among the relevant contraction parameters. The
solution $\ev^{i,j}$ refers to the $j$--th solution in the relevant
list of solutions with $i$ zeros. If not specified, the parameters $a,b$ in contraction
matrices are arbitrary non--zero complex numbers; zeros in
contraction matrices are shown as dots. The subscript $C$ or
$D$ denotes continuous or discrete solution, respectively. The parametric solutions which satisfy \eqref{extend} only if all their parameters are equal to $1$ are marked by $V$ or $\overline{V}$. Non--parametric solutions which do not satisfy \eqref{extend} are marked by $W$ or $\overline{W}$. The parametric solutions which satisfy \eqref{ropa} only if all their parameters are equal to $1$ are marked by $\overline{V}$. Non--parametric solutions which do not satisfy \eqref{ropa} are marked by $\overline{W}$.

\begin{itemize}
\item[$\bullet$\hspace{-10pt}]\hspace{4pt} Trivial solutions
$\ev^{0,1}, \ev^{21,1}$
$$
\left(
\begin{smallmatrix}
\p & 1 & 1 & 1 & 1 & 1 & 1 \\
1 & \p & 1 & 1 & 1 & 1 & 1 \\
1 & 1 & \p & 1 & 1 & 1 & 1 \\
1 & 1 & 1 & \p & 1 & 1 & 1 \\
1 & 1 & 1 & 1 & \p & 1 & 1 \\
1 & 1 & 1 & 1 & 1 & \p & 1 \\
1 & 1 & 1 & 1 & 1 & 1 & \p \\
\end{smallmatrix}
\right)_{\hspace{-5pt} C} \left(
\begin{smallmatrix}
\p & \p & \p & \p & \p & \p & \p \\
\p & \p & \p & \p & \p & \p & \p \\
\p & \p & \p & \p & \p & \p & \p \\
\p & \p & \p & \p & \p & \p & \p \\
\p & \p & \p & \p & \p & \p & \p \\
\p & \p & \p & \p & \p & \p & \p \\
\p & \p & \p & \p & \p & \p & \p \\
\end{smallmatrix}
\right)_{\hspace{-5pt} C}
$$

\item[$\bullet$\hspace{-10pt}]\hspace{4pt} Solutions with 6 zeros
$\ev^{6,1}, \ev^{6,2}$
$$
\left(
\begin{smallmatrix}
\p & 1 & 1 & 1 & 1 & 1 & 1 \\
1 & \p & 1 & 1 & 1 & 1 & 1 \\
1 & 1 & \p & \p & 1 & \p & \p \\
1 & 1 & \p & \p & 1 & \p & \p \\
1 & 1 & 1 & 1 & \p & 1 & 1 \\
1 & 1 & \p & \p & 1 & \p & \p \\
1 & 1 & \p & \p & 1 & \p & \p \\
\end{smallmatrix}
\right)_{\hspace{-5pt} C}\left(
\begin{smallmatrix}
\p & \p & 1 & 1 & 1 & \p & \p \\
\p & \p & 1 & 1 & 1 & \p & \p \\
1 & 1 & \p & 1 & 1 & 1 & 1 \\
1 & 1 & 1 & \p & 1 & 1 & 1 \\
1 & 1 & 1 & 1 & \p & 1 & 1 \\
\p & \p & 1 & 1 & 1 & \p & \p \\
\p & \p & 1 & 1 & 1 & \p & \p \\
\end{smallmatrix}
\right)_{\hspace{-5pt} C}
$$

\item[$\bullet$\hspace{-10pt}]\hspace{4pt} Solution with 9 zeros
$\ev^{9,1}$
$$
\left(
\begin{smallmatrix}
\p & \p & \p & \p & \p & \p & \p \\
\p & \p & -1 & 1 & \p & -1 & 1 \\
\p & -1 & \p & 1 & 1 & -1 & \p \\
\p & 1 & 1 & \p & 1 & \p & 1 \\
\p & \p & 1 & 1 & \p & 1 & 1 \\
\p & -1 & -1 & \p & 1 & \p & 1 \\
\p & 1 & \p & 1 & 1 & 1 & \p \\
\end{smallmatrix}
\right)_{\hspace{-5pt} D}^{\hspace{-5pt} \overline{W}}
$$

\item[$\bullet$\hspace{-10pt}]\hspace{4pt} Solutions with 11 zeros
$\ev^{11,1}, \ev^{11,2}, \ev^{11,3}$
$$
\left(
\begin{smallmatrix}
\p & 1 & 1 & \p & \p & 1 & \p \\
1 & \p & 1 & 1 & 1 & 1 & 1 \\
1 & 1 & \p & \p & 1 & \p & \p \\
\p & 1 & \p & \p & \p & \p & \p \\
\p & 1 & 1 & \p & \p & 1 & \p \\
1 & 1 & \p & \p & 1 & \p & \p \\
\p & 1 & \p & \p & \p & \p & \p \\
\end{smallmatrix}
\right)_{\hspace{-5pt} C} \left(
\begin{smallmatrix}
\p & 1 & 1 & 1 & 1 & 1 & 1 \\
1 & \p & \p & 1 & \p & 1 & \p \\
1 & \p & \p & \p & \p & \p & \p \\
1 & 1 & \p & \p & 1 & \p & \p \\
1 & \p & \p & 1 & \p & 1 & \p \\
1 & 1 & \p & \p & 1 & \p & \p \\
1 & \p & \p & \p & \p & \p & \p \\
\end{smallmatrix}
\right)_{\hspace{-5pt} C} \left(
\begin{smallmatrix}
\p & \p & 1 & \p & \p & \p & \p \\
\p & \p & 1 & 1 & 1 & \p & \p \\
1 & 1 & \p & 1 & 1 & 1 & 1 \\
\p & 1 & 1 & \p & \p & 1 & \p \\
\p & 1 & 1 & \p & \p & 1 & \p \\
\p & \p & 1 & 1 & 1 & \p & \p \\
\p & \p & 1 & \p & \p & \p & \p \\
\end{smallmatrix}
\right)_{\hspace{-5pt} C}
$$

\item[$\bullet$\hspace{-10pt}]\hspace{4pt} Solutions with 12 zeros
$\ev^{12,1}, \ev^{12,2}$
$$
\left(
\begin{smallmatrix}
\p & \p & \p & \p & \p & \p & \p \\
\p & \p & -1 & 1 & 1 & -1 & 1 \\
\p & -1 & \p & \p & 1 & \p & \p \\
\p & 1 & \p & \p & 1 & \p & \p \\
\p & 1 & 1 & 1 & \p & 1 & 1 \\
\p & -1 & \p & \p & 1 & \p & \p \\
\p & 1 & \p & \p & 1 & \p & \p \\
\end{smallmatrix}
\right)_{\hspace{-5pt} D}^{\hspace{-5pt} \overline{W}}\left(
\begin{smallmatrix}
\p & \p & \p & \p & \p & \p & \p \\
\p & \p & -1 & -1 & \p & \p & \p \\
\p & -1 & \p & 1 & -1 & 1 & \p \\
\p & -1 & 1 & \p & -1 & \p & 1 \\
\p & \p & -1 & -1 & \p & 1 & 1 \\
\p & \p & 1 & \p & 1 & \p & \p \\
\p & \p & \p & 1 & 1 & \p & \p \\
\end{smallmatrix}
\right)_{\hspace{-5pt} D}^{\hspace{-5pt} \overline{W}}
$$

\item[$\bullet$\hspace{-10pt}]\hspace{4pt} Solution with 13 zeros
$\ev^{13,1}$
$$
\left(
\begin{smallmatrix}
\p & \p & \p & \p & \p & \p & \p \\
\p & \p & -1 & 1 & \p & -1 & 1 \\
\p & -1 & \p & \p & 1 & \p & \p \\
\p & 1 & \p & \p & 1 & \p & \p \\
\p & \p & 1 & 1 & \p & 1 & 1 \\
\p & -1 & \p & \p & 1 & \p & \p \\
\p & 1 & \p & \p & 1 & \p & \p \\
\end{smallmatrix}
\right)_{\hspace{-5pt} D}^{\hspace{-5pt} \overline{W}}
$$

\item[$\bullet$\hspace{-10pt}]\hspace{4pt} Solutions with 14 zeros
$\ev^{14,1}, \ev^{14,2}$
$$
\left(
\begin{smallmatrix}
\p & \p & \p & \p & \p & \p & \p \\
\p & \p & -1 & 1 & 1 & -1 & 1 \\
\p & -1 & \p & \p & 1 & \p & \p \\
\p & 1 & \p & \p & \p & \p & \p \\
\p & 1 & 1 & \p & \p & 1 & \p \\
\p & -1 & \p & \p & 1 & \p & \p \\
\p & 1 & \p & \p & \p & \p & \p \\
\end{smallmatrix}
\right)_{\hspace{-5pt} D}^{\hspace{-5pt} \overline{W}}\left(
\begin{smallmatrix}
\p & \p & 1 & \p & \p & \p & \p \\
\p & \p & 1 & -1 & \p & \p & \p \\
1 & 1 & \p & 1 & 1 & 1 & \p \\
\p & -1 & 1 & \p & \p & \p & \p \\
\p & \p & 1 & \p & \p & 1 & \p \\
\p & \p & 1 & \p & 1 & \p & \p \\
\p & \p & \p & \p & \p & \p & \p \\
\end{smallmatrix}
\right)_{\hspace{-5pt} D}^{\hspace{-5pt} \overline{W}}
$$

\item[$\bullet$\hspace{-10pt}]\hspace{4pt} Solutions with 15 zeros
$\ev^{15,1},\ldots, \ev^{15,9}$
$$
\left(
\begin{smallmatrix}
\p & a & b & 1 & 1 & 1 & 1 \\
a & \p & \p & \p & \p & \p & \p \\
b & \p & \p & \p & \p & \p & \p \\
1 & \p & \p & \p & \p & \p & \p \\
1 & \p & \p & \p & \p & \p & \p \\
1 & \p & \p & \p & \p & \p & \p \\
1 & \p & \p & \p & \p & \p & \p \\
\end{smallmatrix}
\right)_{\hspace{-5pt} \ast}^{\hspace{-5pt} V}\left(
\begin{smallmatrix}
\p & a & \p & \p & \p & \p & \p \\
a & \p & b & 1 & 1 & 1 & 1 \\
\p & b & \p & \p & \p & \p & \p \\
\p & 1 & \p & \p & \p & \p & \p \\
\p & 1 & \p & \p & \p & \p & \p \\
\p & 1 & \p & \p & \p & \p & \p \\
\p & 1 & \p & \p & \p & \p & \p \\
\end{smallmatrix}
\right)_{\hspace{-5pt} \ast}^{\hspace{-5pt} V}\left(
\begin{smallmatrix}
\p & 1 & 1 & 1 & \p & \p & \p \\
1 & \p & 1 & 1 & \p & \p & \p \\
1 & 1 & \p & 1 & \p & \p & \p \\
1 & 1 & 1 & \p & \p & \p & \p \\
\p & \p & \p & \p & \p & \p & \p \\
\p & \p & \p & \p & \p & \p & \p \\
\p & \p & \p & \p & \p & \p & \p \\
\end{smallmatrix}
\right)_{\hspace{-5pt} C}\left(
\begin{smallmatrix}
\p & 1 & 1 & \p & \p & 1 & \p \\
1 & \p & 1 & \p & \p & \p & 1 \\
1 & 1 & \p & \p & 1 & \p & \p \\
\p & \p & \p & \p & \p & \p & \p \\
\p & \p & 1 & \p & \p & \p & \p \\
1 & \p & \p & \p & \p & \p & \p \\
\p & 1 & \p & \p & \p & \p & \p \\
\end{smallmatrix}
\right)_{\hspace{-5pt} C}\left(
\begin{smallmatrix}
\p & \p & \p & \p & \p & \p & \p \\
\p & \p & 1 & 1 & 1 & \p & \p \\
\p & 1 & \p & 1 & \p & \p & 1 \\
\p & 1 & 1 & \p & \p & 1 & \p \\
\p & 1 & \p & \p & \p & \p & \p \\
\p & \p & \p & 1 & \p & \p & \p \\
\p & \p & 1 & \p & \p & \p & \p \\
\end{smallmatrix}
\right)_{\hspace{-5pt} C}
$$
$$
\left(
\begin{smallmatrix}
\p & \p & \frac{1}{2}(a + 1) & \p & \p & \p & \p \\
\p & \p & a & 1 & 1 & \p & \p \\
\frac{1}{2}(a+1) & a & \p & \p & 1 & \p & \p \\
\p & 1 & \p & \p & \p & \p & \p \\
\p & 1 & 1 & \p & \p & 1 & \p \\
\p & \p & \p & \p & 1 & \p & \p \\
\p & \p & \p & \p & \p & \p & \p \\
\end{smallmatrix}
\right)_{\hspace{-5pt} \dag}^{\hspace{-5pt} \overline{V}}\left(
\begin{smallmatrix}
\p & \p & \p & \p & \p & \p & \p \\
\p & \p & -1 & 1 & \p & -1 & 1 \\
\p & -1 & \p & \p & 1 & \p & \p \\
\p & 1 & \p & \p & \p & \p & \p \\
\p & \p & 1 & \p & \p & 1 & \p \\
\p & -1 & \p & \p & 1 & \p & \p \\
\p & 1 & \p & \p & \p & \p & \p \\
\end{smallmatrix}
\right)_{\hspace{-5pt} D}^{\hspace{-5pt} \overline{W}}\left(
\begin{smallmatrix}
\p & \p & \p & \p & \p & \p & \p \\
\p & \p & 1 & -1 & \p & \p & \p \\
\p & 1 & \p & 1 & 1 & 1 & \p \\
\p & -1 & 1 & \p & \p & \p & \p \\
\p & \p & 1 & \p & \p & 1 & \p \\
\p & \p & 1 & \p & 1 & \p & \p \\
\p & \p & \p & \p & \p & \p & \p \\
\end{smallmatrix}
\right)_{\hspace{-5pt} D}^{\hspace{-5pt} \overline{W}} \left(
\begin{smallmatrix}
\p & \p & \p & \p & \p & \p & \p \\
\p & \p & \p & 1 & 1 & 1 & \p \\
\p & \p & \p & \p & \p & \p & \p \\
\p & 1 & \p & \p & 1 & 1 & \p \\
\p & 1 & \p & 1 & \p & 1 & \p \\
\p & 1 & \p & 1 & 1 & \p & \p \\
\p & \p & \p & \p & \p & \p & \p \\
\end{smallmatrix}
\right)_{\hspace{-5pt} C}
$$
{\small $^\ast$ Continuous for $a=b=1$, otherwise discrete. }
{\small $^\dag$ $a\neq 0,-1$, continuous for $a=1$, otherwise
discrete. }

\item[$\bullet$\hspace{-10pt}]\hspace{4pt} Solutions with 16 zeros
$\ev^{16,1},\ldots, \ev^{16,12}$
$$
\left(
\begin{smallmatrix}
\p & a & 1 & 1 & 1 & 1 & \p \\
a & \p & \p & \p & \p & \p & \p \\
1 & \p & \p & \p & \p & \p & \p \\
1 & \p & \p & \p & \p & \p & \p \\
1 & \p & \p & \p & \p & \p & \p \\
1 & \p & \p & \p & \p & \p & \p \\
\p & \p & \p & \p & \p & \p & \p \\
\end{smallmatrix}
\right)_{\hspace{-5pt} D}^{\hspace{-5pt} W} \left(
\begin{smallmatrix}
\p & 1 & 1 & 1 & \p & \p & \p \\
1 & \p & 1 & 1 & \p & \p & \p \\
1 & 1 & \p & \p & \p & \p & \p \\
1 & 1 & \p & \p & \p & \p & \p \\
\p & \p & \p & \p & \p & \p & \p \\
\p & \p & \p & \p & \p & \p & \p \\
\p & \p & \p & \p & \p & \p & \p \\
\end{smallmatrix}
\right)_{\hspace{-5pt} C}\left(
\begin{smallmatrix}
\p & a & \p & \p & \p & \p & \p \\
a & \p & 1 & 1 & 1 & 1 & \p \\
\p & 1 & \p & \p & \p & \p & \p \\
\p & 1 & \p & \p & \p & \p & \p \\
\p & 1 & \p & \p & \p & \p & \p \\
\p & 1 & \p & \p & \p & \p & \p \\
\p & \p & \p & \p & \p & \p & \p \\
\end{smallmatrix}
\right)_{\hspace{-5pt} D}^{\hspace{-5pt} W}\left(
\begin{smallmatrix}
\p & 1 & \p & \p & \p & \p & \p \\
1 & \p & a & 1 & \p & 1 & 1 \\
\p & a & \p & \p & \p & \p & \p \\
\p & 1 & \p & \p & \p & \p & \p \\
\p & \p & \p & \p & \p & \p & \p \\
\p & 1 & \p & \p & \p & \p & \p \\
\p & 1 & \p & \p & \p & \p & \p \\
\end{smallmatrix}
\right)_{\hspace{-5pt} D}^{\hspace{-5pt} W}\left(
\begin{smallmatrix}
\p & \p & \p & \p & \p & \p & \p \\
\p & \p & a & 1 & 1 & 1 & 1 \\
\p & a & \p & \p & \p & \p & \p \\
\p & 1 & \p & \p & \p & \p & \p \\
\p & 1 & \p & \p & \p & \p & \p \\
\p & 1 & \p & \p & \p & \p & \p \\
\p & 1 & \p & \p & \p & \p & \p \\
\end{smallmatrix}
\right)_{\hspace{-5pt} D}^{\hspace{-5pt} W}
$$
$$
\left(
\begin{smallmatrix}
\p & 1 & 1 & \p & \p & \p & \p \\
1 & \p & 1 & 1 & \p & \p & \p \\
1 & 1 & \p & 1 & \p & \p & \p \\
\p & 1 & 1 & \p & \p & \p & \p \\
\p & \p & \p & \p & \p & \p & \p \\
\p & \p & \p & \p & \p & \p & \p \\
\p & \p & \p & \p & \p & \p & \p \\
\end{smallmatrix}
\right)_{\hspace{-5pt} C} \left(
\begin{smallmatrix}
\p & \p & \frac{1}{2} & \p & \p & \p & \p \\
\p & \p & 1 & 1 & 1 & \p & \p \\
\frac {1}{2} & 1 & \p & \p & 1 & \p & \p \\
\p & 1 & \p & \p & \p & \p & \p \\
\p & 1 & 1 & \p & \p & \p & \p \\
\p & \p & \p & \p & \p & \p & \p \\
\p & \p & \p & \p & \p & \p & \p \\
\end{smallmatrix}
\right)_{\hspace{-5pt} D}^{\hspace{-5pt} \overline{W}}\left(
\begin{smallmatrix}
\p & \p & \p & \p & \p & \frac{1}{2}  & \p \\
\p & \p & \p & \p & 1 & 1 & 1 \\
\p & \p & \p & \p & 1 & \p & \p \\
\p & \p & \p & \p & \p & \p & \p \\
\p & 1 & 1 & \p & \p & \p & \p \\
\frac{1}{2} & 1 & \p & \p & \p & \p & \p \\
\p & 1 & \p & \p & \p & \p & \p \\
\end{smallmatrix}
\right)_{\hspace{-5pt} D}^{\hspace{-5pt} \overline{W}}\left(
\begin{smallmatrix}
\p & \p & \p & \p & \p & \p & \p \\
\p & \p & 1 & \p & 1 & \p & 1 \\
\p & 1 & \p & \p & 1 & \p & 1 \\
\p & \p & \p & \p & \p & \p & \p \\
\p & 1 & 1 & \p & \p & \p & \p \\
\p & \p & \p & \p & \p & \p & \p \\
\p & 1 & 1 & \p & \p & \p & \p \\
\end{smallmatrix}
\right)_{\hspace{-5pt} C}\left(
\begin{smallmatrix}
\p & \p & 1 & \p & \p & \p & \p \\
\p & \p & -1 & 1 & \p & \p & \p \\
1 & -1 & \p & \p & 1 & \p & \p \\
\p & 1 & \p & \p & \p & \p & \p \\
\p & \p & 1 & \p & \p & 1 & \p \\
\p & \p & \p & \p & 1 & \p & \p \\
\p & \p & \p & \p & \p & \p & \p \\
\end{smallmatrix}
\right)_{\hspace{-5pt} D}^{\hspace{-5pt} \overline{W}}
$$
$$
\left(
\begin{smallmatrix}
\p & \p & \p & \p & \p & \p & \p \\
\p & \p & -1 & 1 & 1 & \p & \p \\
\p & -1 & \p & \p & 1 & \p & \p \\
\p & 1 & \p & \p & \p & \p & \p \\
\p & 1 & 1 & \p & \p & 1 & \p \\
\p & \p & \p & \p & 1 & \p & \p \\
\p & \p & \p & \p & \p & \p & \p \\
\end{smallmatrix}
\right)_{\hspace{-5pt} D}^{\hspace{-5pt} \overline{W}}\left(
\begin{smallmatrix}
\p & \p & \p & \p & \p & \p & \p \\
\p & \p & \p & 1 & 1 & 1 & \p \\
\p & \p & \p & \p & \p & \p & \p \\
\p & 1 & \p & \p & 1 & \p & \p \\
\p & 1 & \p & 1 & \p & 1 & \p \\
\p & 1 & \p & \p & 1 & \p & \p \\
\p & \p & \p & \p & \p & \p & \p \\
\end{smallmatrix}
\right)_{\hspace{-5pt} C}
$$

\item[$\bullet$\hspace{-10pt}]\hspace{4pt} Solutions with 17 zeros
$\ev^{17,1},\ldots, \ev^{17,18}$
$$
\left(
\begin{smallmatrix}
\p & 1 & 1 & 1 & 1 & \p & \p \\
1 & \p & \p & \p & \p & \p & \p \\
1 & \p & \p & \p & \p & \p & \p \\
1 & \p & \p & \p & \p & \p & \p \\
1 & \p & \p & \p & \p & \p & \p \\
\p & \p & \p & \p & \p & \p & \p \\
\p & \p & \p & \p & \p & \p & \p \\
\end{smallmatrix}
\right)_{\hspace{-5pt} D}^{\hspace{-5pt} W}\left(
\begin{smallmatrix}
\p & a & \p & 1 & 1 & 1 & \p \\
a & \p & \p & \p & \p & \p & \p \\
\p & \p & \p & \p & \p & \p & \p \\
1 & \p & \p & \p & \p & \p & \p \\
1 & \p & \p & \p & \p & \p & \p \\
1 & \p & \p & \p & \p & \p & \p \\
\p & \p & \p & \p & \p & \p & \p \\
\end{smallmatrix}
\right)_{\hspace{-5pt} D}^{\hspace{-5pt} V}\left(
\begin{smallmatrix}
\p & 1 & 1 & 1 & \p & \p & \p \\
1 & \p & 1 & \p & \p & \p & \p \\
1 & 1 & \p & \p & \p & \p & \p \\
1 & \p & \p & \p & \p & \p & \p \\
\p & \p & \p & \p & \p & \p & \p \\
\p & \p & \p & \p & \p & \p & \p \\
\p & \p & \p & \p & \p & \p & \p \\
\end{smallmatrix}
\right)_{\hspace{-5pt} C}\left(
\begin{smallmatrix}
\p & 1 & 1 & \p & \p & \p & \p \\
1 & \p & 1 & 1 & \p & \p & \p \\
1 & 1 & \p & \p & \p & \p & \p \\
\p & 1 & \p & \p & \p & \p & \p \\
\p & \p & \p & \p & \p & \p & \p \\
\p & \p & \p & \p & \p & \p & \p \\
\p & \p & \p & \p & \p & \p & \p \\
\end{smallmatrix}
\right)_{\hspace{-5pt} C}\left(
\begin{smallmatrix}
\p & \p & 1 & 1 & \p & \p & \p \\
\p & \p & 1 & 1 & \p & \p & \p \\
1 & 1 & \p & \p & \p & \p & \p \\
1 & 1 & \p & \p & \p & \p & \p \\
\p & \p & \p & \p & \p & \p & \p \\
\p & \p & \p & \p & \p & \p & \p \\
\p & \p & \p & \p & \p & \p & \p \\
\end{smallmatrix}
\right)_{\hspace{-5pt} C}\left(
\begin{smallmatrix}
\p & 1 & \p & \p & \p & \p & \p \\
1 & \p & 1 & 1 & 1 & \p & \p \\
\p & 1 & \p & \p & \p & \p & \p \\
\p & 1 & \p & \p & \p & \p & \p \\
\p & 1 & \p & \p & \p & \p & \p \\
\p & \p & \p & \p & \p & \p & \p \\
\p & \p & \p & \p & \p & \p & \p \\
\end{smallmatrix}
\right)_{\hspace{-5pt} D}^{\hspace{-5pt} W}
$$
$$
\left(
\begin{smallmatrix}
\p & 1 & \p & \p & \p & \p & \p \\
1 & \p & 1 & 1 & \p & 1 & \p \\
\p & 1 & \p & \p & \p & \p & \p \\
\p & 1 & \p & \p & \p & \p & \p \\
\p & \p & \p & \p & \p & \p & \p \\
\p & 1 & \p & \p & \p & \p & \p \\
\p & \p & \p & \p & \p & \p & \p \\
\end{smallmatrix}
\right)_{\hspace{-5pt} D}^{\hspace{-5pt} W} \left(
\begin{smallmatrix}
\p & a & \p & \p & \p & \p & \p \\
a & \p & 1 & \p & 1 & 1 & \p \\
\p & 1 & \p & \p & \p & \p & \p \\
\p & \p & \p & \p & \p & \p & \p \\
\p & 1 & \p & \p & \p & \p & \p \\
\p & 1 & \p & \p & \p & \p & \p \\
\p & \p & \p & \p & \p & \p & \p \\
\end{smallmatrix}
\right)_{\hspace{-5pt} D}^{\hspace{-5pt} V}\left(
\begin{smallmatrix}
\p & 1 & \p & \p & \p & \p & \p \\
1 & \p & \p & 1 & 1 & 1 & \p \\
\p & \p & \p & \p & \p & \p & \p \\
\p & 1 & \p & \p & \p & \p & \p \\
\p & 1 & \p & \p & \p & \p & \p \\
\p & 1 & \p & \p & \p & \p & \p \\
\p & \p & \p & \p & \p & \p & \p \\
\end{smallmatrix}
\right)_{\hspace{-5pt} D}^{\hspace{-5pt} W}\left(
\begin{smallmatrix}
\p & \p & \p & \p & \p & \p & \p \\
\p & \p & 1 & 1 & 1 & 1 & \p \\
\p & 1 & \p & \p & \p & \p & \p \\
\p & 1 & \p & \p & \p & \p & \p \\
\p & 1 & \p & \p & \p & \p & \p \\
\p & 1 & \p & \p & \p & \p & \p \\
\p & \p & \p & \p & \p & \p & \p \\
\end{smallmatrix}
\right)_{\hspace{-5pt} D}^{\hspace{-5pt} W}\left(
\begin{smallmatrix}
\p & \p & \p & \p & \p & \p & \p \\
\p & \p & a & 1 & \p & 1 & 1 \\
\p & a & \p & \p & \p & \p & \p \\
\p & 1 & \p & \p & \p & \p & \p \\
\p & \p & \p & \p & \p & \p & \p \\
\p & 1 & \p & \p & \p & \p & \p \\
\p & 1 & \p & \p & \p & \p & \p \\
\end{smallmatrix}
\right)_{\hspace{-5pt} D}^{\hspace{-5pt} V}\left(
\begin{smallmatrix}
\p & 1 & \p & \p & \p & \p & \p \\
1 & \p & 1 & 1 & \p & \p & \p \\
\p & 1 & \p & 1 & \p & \p & \p \\
\p & 1 & 1 & \p & \p & \p & \p \\
\p & \p & \p & \p & \p & \p & \p \\
\p & \p & \p & \p & \p & \p & \p \\
\p & \p & \p & \p & \p & \p & \p \\
\end{smallmatrix}
\right)_{\hspace{-5pt} C}
$$
$$
\left(
\begin{smallmatrix}
\p & \p & \frac{1}{2} & \p & \p & \p & \p \\
\p & \p & \p & 1 & 1 & \p & \p \\
\frac{1}{2} & \p & \p & \p & 1 & \p & \p \\
\p & 1 & \p & \p & \p & \p & \p \\
\p & 1 & 1 & \p & \p & \p & \p \\
\p & \p & \p & \p & \p & \p & \p \\
\p & \p & \p & \p & \p & \p & \p \\
\end{smallmatrix}
\right)_{\hspace{-5pt} D}^{\hspace{-5pt} \overline{W}}\left(
\begin{smallmatrix}
\p & \p & \p & \p & \p & \p & \p \\
\p & \p & 1 & \p & 1 & \p & 1 \\
\p & 1 & \p & \p & 1 & \p & \p \\
\p & \p & \p & \p & \p & \p & \p \\
\p & 1 & 1 & \p & \p & \p & \p \\
\p & \p & \p & \p & \p & \p & \p \\
\p & 1 & \p & \p & \p & \p & \p \\
\end{smallmatrix}
\right)_{\hspace{-5pt} C} \left(
\begin{smallmatrix}
\p & \p & \p & \p & \p & \p & \p \\
\p & \p & 1 & \p & 1 & \p & 1 \\
\p & 1 & \p & \p & \p & \p & 1 \\
\p & \p & \p & \p & \p & \p & \p \\
\p & 1 & \p & \p & \p & \p & \p \\
\p & \p & \p & \p & \p & \p & \p \\
\p & 1 & 1 & \p & \p & \p & \p \\
\end{smallmatrix}
\right)_{\hspace{-5pt} C}\left(
\begin{smallmatrix}
\p & \p & \p & \p & \p & \p & \p \\
\p & \p & \p & \p & 1 & \p & 1 \\
\p & \p & \p & \p & 1 & \p & 1 \\
\p & \p & \p & \p & \p & \p & \p \\
\p & 1 & 1 & \p & \p & \p & \p \\
\p & \p & \p & \p & \p & \p & \p \\
\p & 1 & 1 & \p & \p & \p & \p \\
\end{smallmatrix}
\right)_{\hspace{-5pt} C}\left(
\begin{smallmatrix}
\p & \p & \p & \p & \p & \p & \p \\
\p & \p & \operatorname{-1} & 1 & \p & \p & \p \\
\p & \operatorname{-1} & \p & \p & 1 & \p & \p \\
\p & 1 & \p & \p & \p & \p & \p \\
\p & \p & 1 & \p & \p & 1 & \p \\
\p & \p & \p & \p & 1 & \p & \p \\
\p & \p & \p & \p & \p & \p & \p \\
\end{smallmatrix}
\right)_{\hspace{-5pt} D}^{\hspace{-5pt} \overline{W}}\left(
\begin{smallmatrix}
\p & \p & \p & \p & \p & \p & \p \\
\p & \p & \p & 1 & \p & 1 & \p \\
\p & \p & \p & \p & \p & \p & \p \\
\p & 1 & \p & \p & 1 & \p & \p \\
\p & \p & \p & 1 & \p & 1 & \p \\
\p & 1 & \p & \p & 1 & \p & \p \\
\p & \p & \p & \p & \p & \p & \p \\
\end{smallmatrix}
\right)_{\hspace{-5pt} C}
$$

\item[$\bullet$\hspace{-10pt}]\hspace{4pt} Solutions with 18 zeros
$\ev^{18,1},\ldots, \ev^{18,23}$
$$
\left(
\begin{smallmatrix}
\p & 1 & 1 & 1 & \p & \p & \p \\
1 & \p & \p & \p & \p & \p & \p \\
1 & \p & \p & \p & \p & \p & \p \\
1 & \p & \p & \p & \p & \p & \p \\
\p & \p & \p & \p & \p & \p & \p \\
\p & \p & \p & \p & \p & \p & \p \\
\p & \p & \p & \p & \p & \p & \p \\
\end{smallmatrix}
\right)_{\hspace{-5pt} C}\left(
\begin{smallmatrix}
\p & 1 & 1 & \p & 1 & \p & \p \\
1 & \p & \p & \p & \p & \p & \p \\
1 & \p & \p & \p & \p & \p & \p \\
\p & \p & \p & \p & \p & \p & \p \\
1 & \p & \p & \p & \p & \p & \p \\
\p & \p & \p & \p & \p & \p & \p \\
\p & \p & \p & \p & \p & \p & \p \\
\end{smallmatrix}
\right)_{\hspace{-5pt} D}^{\hspace{-5pt} W}\left(
\begin{smallmatrix}
\p & \p & 1 & 1 & 1 & \p & \p \\
\p & \p & \p & \p & \p & \p & \p \\
1 & \p & \p & \p & \p & \p & \p \\
1 & \p & \p & \p & \p & \p & \p \\
1 & \p & \p & \p & \p & \p & \p \\
\p & \p & \p & \p & \p & \p & \p \\
\p & \p & \p & \p & \p & \p & \p \\
\end{smallmatrix}
\right)_{\hspace{-5pt} C}\left(
\begin{smallmatrix}
\p & 1 & 1 & \p & \p & \p & \p \\
1 & \p & 1 & \p & \p & \p & \p \\
1 & 1 & \p & \p & \p & \p & \p \\
\p & \p & \p & \p & \p & \p & \p \\
\p & \p & \p & \p & \p & \p & \p \\
\p & \p & \p & \p & \p & \p & \p \\
\p & \p & \p & \p & \p & \p & \p \\
\end{smallmatrix}
\right)_{\hspace{-5pt} C}\left(
\begin{smallmatrix}
\p & 1 & \p & 1 & \p & \p & \p \\
1 & \p & 1 & \p & \p & \p & \p \\
\p & 1 & \p & \p & \p & \p & \p \\
1 & \p & \p & \p & \p & \p & \p \\
\p & \p & \p & \p & \p & \p & \p \\
\p & \p & \p & \p & \p & \p & \p \\
\p & \p & \p & \p & \p & \p & \p \\
\end{smallmatrix}
\right)_{\hspace{-5pt} C}\left(
\begin{smallmatrix}
\p & 1 & \p & \p & \p & \p & \p \\
1 & \p & 1 & 1 & \p & \p & \p \\
\p & 1 & \p & \p & \p & \p & \p \\
\p & 1 & \p & \p & \p & \p & \p \\
\p & \p & \p & \p & \p & \p & \p \\
\p & \p & \p & \p & \p & \p & \p \\
\p & \p & \p & \p & \p & \p & \p \\
\end{smallmatrix}
\right)_{\hspace{-5pt} C}
$$
$$
\left(
\begin{smallmatrix}
\p & \p & 1 & \p & \p & \p & \p \\
\p & \p & 1 & 1 & \p & \p & \p \\
1 & 1 & \p & \p & \p & \p & \p \\
\p & 1 & \p & \p & \p & \p & \p \\
\p & \p & \p & \p & \p & \p & \p \\
\p & \p & \p & \p & \p & \p & \p \\
\p & \p & \p & \p & \p & \p & \p \\
\end{smallmatrix}
\right)_{\hspace{-5pt} C}\left(
\begin{smallmatrix}
\p & 1 & \p & \p & 1 & \p & \p \\
1 & \p & \p & \p & 1 & \p & \p \\
\p & \p & \p & \p & \p & \p & \p \\
\p & \p & \p & \p & \p & \p & \p \\
1 & 1 & \p & \p & \p & \p & \p \\
\p & \p & \p & \p & \p & \p & \p \\
\p & \p & \p & \p & \p & \p & \p \\
\end{smallmatrix}
\right)_{\hspace{-5pt} D}\left(
\begin{smallmatrix}
\p & 1 & \p & \p & \p & \p & \p \\
1 & \p & 1 & \p & 1 & \p & \p \\
\p & 1 & \p & \p & \p & \p & \p \\
\p & \p & \p & \p & \p & \p & \p \\
\p & 1 & \p & \p & \p & \p & \p \\
\p & \p & \p & \p & \p & \p & \p \\
\p & \p & \p & \p & \p & \p & \p \\
\end{smallmatrix}
\right)_{\hspace{-5pt} D}^{\hspace{-5pt} W}\left(
\begin{smallmatrix}
\p & \p & \p & \p & \p & \p & \p \\
\p & \p & 1 & 1 & 1 & \p & \p \\
\p & 1 & \p & \p & \p & \p & \p \\
\p & 1 & \p & \p & \p & \p & \p \\
\p & 1 & \p & \p & \p & \p & \p \\
\p & \p & \p & \p & \p & \p & \p \\
\p & \p & \p & \p & \p & \p & \p \\
\end{smallmatrix}
\right)_{\hspace{-5pt} C}\left(
\begin{smallmatrix}
\p & 1 & \p & \p & \p & \p & \p \\
1 & \p & 1 & \p & \p & 1 & \p \\
\p & 1 & \p & \p & \p & \p & \p \\
\p & \p & \p & \p & \p & \p & \p \\
\p & \p & \p & \p & \p & \p & \p \\
\p & 1 & \p & \p & \p & \p & \p \\
\p & \p & \p & \p & \p & \p & \p \\
\end{smallmatrix}
\right)_{\hspace{-5pt} D}^{\hspace{-5pt} W}\left(
\begin{smallmatrix}
\p & 1 & \p & \p & \p & \p & \p \\
1 & \p & \p & 1 & \p & 1 & \p \\
\p & \p & \p & \p & \p & \p & \p \\
\p & 1 & \p & \p & \p & \p & \p \\
\p & \p & \p & \p & \p & \p & \p \\
\p & 1 & \p & \p & \p & \p & \p \\
\p & \p & \p & \p & \p & \p & \p \\
\end{smallmatrix}
\right)_{\hspace{-5pt} C}
$$
$$
\left(
\begin{smallmatrix}
\p & \p & \p & \p & \p & \p & \p \\
\p & \p & 1 & 1 & \p & 1 & \p \\
\p & 1 & \p & \p & \p & \p & \p \\
\p & 1 & \p & \p & \p & \p & \p \\
\p & \p & \p & \p & \p & \p & \p \\
\p & 1 & \p & \p & \p & \p & \p \\
\p & \p & \p & \p & \p & \p & \p \\
\end{smallmatrix}
\right)_{\hspace{-5pt} D}^{\hspace{-5pt} W}\left(
\begin{smallmatrix}
\p & \p & \p & \p & \p & \p & \p \\
\p & \p & 1 & \p & 1 & 1 & \p \\
\p & 1 & \p & \p & \p & \p & \p \\
\p & \p & \p & \p & \p & \p & \p \\
\p & 1 & \p & \p & \p & \p & \p \\
\p & 1 & \p & \p & \p & \p & \p \\
\p & \p & \p & \p & \p & \p & \p \\
\end{smallmatrix}
\right)_{\hspace{-5pt} D}^{\hspace{-5pt} W}\left(
\begin{smallmatrix}
\p & \p & \p & \p & \p & \p & \p \\
\p & \p & \p & 1 & 1 & 1 & \p \\
\p & \p & \p & \p & \p & \p & \p \\
\p & 1 & \p & \p & \p & \p & \p \\
\p & 1 & \p & \p & \p & \p & \p \\
\p & 1 & \p & \p & \p & \p & \p \\
\p & \p & \p & \p & \p & \p & \p \\
\end{smallmatrix}
\right)_{\hspace{-5pt} C}\left(
\begin{smallmatrix}
\p & \p & \p & \p & \p & \p & \p \\
\p & \p & 1 & 1 & \p & \p & \p \\
\p & 1 & \p & 1 & \p & \p & \p \\
\p & 1 & 1 & \p & \p & \p & \p \\
\p & \p & \p & \p & \p & \p & \p \\
\p & \p & \p & \p & \p & \p & \p \\
\p & \p & \p & \p & \p & \p & \p
\end{smallmatrix}
\right)_{\hspace{-5pt} C}\left(
\begin{smallmatrix}
\p & \p & \p & \p & \p & \p & \p \\
\p & \p & 1 & \p & 1 & \p & \p \\
\p & 1 & \p & \p & 1 & \p & \p \\
\p & \p & \p & \p & \p & \p & \p \\
\p & 1 & 1 & \p & \p & \p & \p \\
\p & \p & \p & \p & \p & \p & \p \\
\p & \p & \p & \p & \p & \p & \p \\
\end{smallmatrix}
\right)_{\hspace{-5pt} C}\left(
\begin{smallmatrix}
\p & \p & \p & \p & \p & 1 & \p \\
\p & \p & \p & \p & \p & \p & 1 \\
\p & \p & \p & \p & 1 & \p & \p \\
\p & \p & \p & \p & \p & \p & \p \\
\p & \p & 1 & \p & \p & \p & \p \\
1 & \p & \p & \p & \p & \p & \p \\
\p & 1 & \p & \p & \p & \p & \p \\
\end{smallmatrix}
\right)_{\hspace{-5pt} C}
$$
$$
\left(
\begin{smallmatrix}
\p & \p & \p & \p & \p & \p & \p \\
\p & \p & 1 & \p & \p & \p & 1 \\
\p & 1 & \p & \p & 1 & \p & \p \\
\p & \p & \p & \p & \p & \p & \p \\
\p & \p & 1 & \p & \p & \p & \p \\
\p & \p & \p & \p & \p & \p & \p \\
\p & 1 & \p & \p & \p & \p & \p \\
\end{smallmatrix}
\right)_{\hspace{-5pt} C}\left(
\begin{smallmatrix}
\p & \p & \p & \p & \p & \p & \p \\
\p & \p & \p & \p & 1 & \p & 1 \\
\p & \p & \p & \p & 1 & \p & \p \\
\p & \p & \p & \p & \p & \p & \p \\
\p & 1 & 1 & \p & \p & \p & \p \\
\p & \p & \p & \p & \p & \p & \p \\
\p & 1 & \p & \p & \p & \p & \p \\
\end{smallmatrix}
\right)_{\hspace{-5pt} C}\left(
\begin{smallmatrix}
\p & \p & \p & \p & \p & \p & \p \\
\p & \p & 1 & \p & \p & 1 & \p \\
\p & 1 & \p & \p & \p & 1 & \p \\
\p & \p & \p & \p & \p & \p & \p \\
\p & \p & \p & \p & \p & \p & \p \\
\p & 1 & 1 & \p & \p & \p & \p \\
\p & \p & \p & \p & \p & \p & \p \\
\end{smallmatrix}
\right)_{\hspace{-5pt} D}\left(
\begin{smallmatrix}
\p & \p & \p & \p & \p & \p & \p \\
\p & \p & 1 & \p & 1 & \p & \p \\
\p & 1 & \p & \p & \p & \p & 1 \\
\p & \p & \p & \p & \p & \p & \p \\
\p & 1 & \p & \p & \p & \p & \p \\
\p & \p & \p & \p & \p & \p & \p \\
\p & \p & 1 & \p & \p & \p & \p \\
\end{smallmatrix}
\right)_{\hspace{-5pt} C}\left(
\begin{smallmatrix}
\p & \p & \p & \p & \p & \p & \p \\
\p & \p & \p & \p & 1 & \p & \p \\
\p & \p & \p & \p & \p & \p & 1 \\
\p & \p & \p & \p & \p & 1 & \p \\
\p & 1 & \p & \p & \p & \p & \p \\
\p & \p & \p & 1 & \p & \p & \p \\
\p & \p & 1 & \p & \p & \p & \p \\
\end{smallmatrix}
\right)_{\hspace{-5pt} C}
$$

\item[$\bullet$\hspace{-10pt}]\hspace{4pt} Solutions with 19 zeros
$\ev^{19,1},\ldots, \ev^{19,11}$
$$
\left(
\begin{smallmatrix}
\p & 1 & 1 & \p & \p & \p & \p \\
1 & \p & \p & \p & \p & \p & \p \\
1 & \p & \p & \p & \p & \p & \p \\
\p & \p & \p & \p & \p & \p & \p \\
\p & \p & \p & \p & \p & \p & \p \\
\p & \p & \p & \p & \p & \p & \p \\
\p & \p & \p & \p & \p & \p & \p \\
\end{smallmatrix}
\right)_{\hspace{-5pt} C}\left(
\begin{smallmatrix}
\p & 1 & \p & \p & 1 & \p & \p \\
1 & \p & \p & \p & \p & \p & \p \\
\p & \p & \p & \p & \p & \p & \p \\
\p & \p & \p & \p & \p & \p & \p \\
1 & \p & \p & \p & \p & \p & \p \\
\p & \p & \p & \p & \p & \p & \p \\
\p & \p & \p & \p & \p & \p & \p \\
\end{smallmatrix}
\right)_{\hspace{-5pt} D}\left(
\begin{smallmatrix}
\p & 1 & \p & \p & \p & \p & \p \\
1 & \p & 1 & \p & \p & \p & \p \\
\p & 1 & \p & \p & \p & \p & \p \\
\p & \p & \p & \p & \p & \p & \p \\
\p & \p & \p & \p & \p & \p & \p \\
\p & \p & \p & \p & \p & \p & \p \\
\p & \p & \p & \p & \p & \p & \p \\
\end{smallmatrix}
\right)_{\hspace{-5pt} C}\left(
\begin{smallmatrix}
\p & \p & \p & 1 & \p & \p & \p \\
\p & \p & 1 & \p & \p & \p & \p \\
\p & 1 & \p & \p & \p & \p & \p \\
1 & \p & \p & \p & \p & \p & \p \\
\p & \p & \p & \p & \p & \p & \p \\
\p & \p & \p & \p & \p & \p & \p \\
\p & \p & \p & \p & \p & \p & \p \\
\end{smallmatrix}
\right)_{\hspace{-5pt} C}\left(
\begin{smallmatrix}
\p & \p & \p & \p & \p & \p & \p \\
\p & \p & 1 & 1 & \p & \p & \p \\
\p & 1 & \p & \p & \p & \p & \p \\
\p & 1 & \p & \p & \p & \p & \p \\
\p & \p & \p & \p & \p & \p & \p \\
\p & \p & \p & \p & \p & \p & \p \\
\p & \p & \p & \p & \p & \p & \p \\
\end{smallmatrix}
\right)_{\hspace{-5pt} C}\left(
\begin{smallmatrix}
\p & 1 & \p & \p & \p & \p & \p \\
1 & \p & \p & \p & 1 & \p & \p \\
\p & \p & \p & \p & \p & \p & \p \\
\p & \p & \p & \p & \p & \p & \p \\
\p & 1 & \p & \p & \p & \p & \p \\
\p & \p & \p & \p & \p & \p & \p \\
\p & \p & \p & \p & \p & \p & \p \\
\end{smallmatrix}
\right)_{\hspace{-5pt} D}
$$
$$
\left(
\begin{smallmatrix}
\p & \p & \p & \p & \p & \p & \p \\
\p & \p & 1 & \p & 1 & \p & \p \\
\p & 1 & \p & \p & \p & \p & \p \\
\p & \p & \p & \p & \p & \p & \p \\
\p & 1 & \p & \p & \p & \p & \p \\
\p & \p & \p & \p & \p & \p & \p \\
\p & \p & \p & \p & \p & \p & \p \\
\end{smallmatrix}
\right)_{\hspace{-5pt} C}\left(
\begin{smallmatrix}
\p & \p & \p & \p & \p & \p & \p \\
\p & \p & 1 & \p & \p & 1 & \p \\
\p & 1 & \p & \p & \p & \p & \p \\
\p & \p & \p & \p & \p & \p & \p \\
\p & \p & \p & \p & \p & \p & \p \\
\p & 1 & \p & \p & \p & \p & \p \\
\p & \p & \p & \p & \p & \p & \p \\
\end{smallmatrix}
\right)_{\hspace{-5pt} D}\left(
\begin{smallmatrix}
\p & \p & \p & \p & \p & \p & \p \\
\p & \p & \p & 1 & \p & 1 & \p \\
\p & \p & \p & \p & \p & \p & \p \\
\p & 1 & \p & \p & \p & \p & \p \\
\p & \p & \p & \p & \p & \p & \p \\
\p & 1 & \p & \p & \p & \p & \p \\
\p & \p & \p & \p & \p & \p & \p \\
\end{smallmatrix}
\right)_{\hspace{-5pt} C}\left(
\begin{smallmatrix}
\p & \p & \p & \p & \p & \p & \p \\
\p & \p & \p & \p & \p & \p & 1 \\
\p & \p & \p & \p & 1 & \p & \p \\
\p & \p & \p & \p & \p & \p & \p \\
\p & \p & 1 & \p & \p & \p & \p \\
\p & \p & \p & \p & \p & \p & \p \\
\p & 1 & \p & \p & \p & \p & \p \\
\end{smallmatrix}
\right)_{\hspace{-5pt} C}\left(
\begin{smallmatrix}
\p & \p & \p & \p & \p & \p & \p \\
\p & \p & \p & \p & 1 & \p & \p \\
\p & \p & \p & \p & \p & \p & 1 \\
\p & \p & \p & \p & \p & \p & \p \\
\p & 1 & \p & \p & \p & \p & \p \\
\p & \p & \p & \p & \p & \p & \p \\
\p & \p & 1 & \p & \p & \p & \p \\
\end{smallmatrix}
\right)_{\hspace{-5pt} C}
$$

\item[$\bullet$\hspace{-10pt}]\hspace{4pt} Solutions with 20 zeros
$\ev^{20,1}, \ev^{20,2}, \ev^{20,3}$
$$
\left(
\begin{smallmatrix}
\p & 1 & \p & \p & \p & \p & \p \\
1 & \p & \p & \p & \p & \p & \p \\
\p & \p & \p & \p & \p & \p & \p \\
\p & \p & \p & \p & \p & \p & \p \\
\p & \p & \p & \p & \p & \p & \p \\
\p & \p & \p & \p & \p & \p & \p \\
\p & \p & \p & \p & \p & \p & \p \\
\end{smallmatrix}
\right)_{\hspace{-5pt} C}\left(
\begin{smallmatrix}
\p & \p & \p & \p & \p & \p & \p \\
\p & \p & 1 & \p & \p & \p & \p \\
\p & 1 & \p & \p & \p & \p & \p \\
\p & \p & \p & \p & \p & \p & \p \\
\p & \p & \p & \p & \p & \p & \p \\
\p & \p & \p & \p & \p & \p & \p \\
\p & \p & \p & \p & \p & \p & \p \\
\end{smallmatrix}
\right)_{\hspace{-5pt} C}\left(
\begin{smallmatrix}
\p & \p & \p & \p & \p & \p & \p \\
\p & \p & \p & \p & 1 & \p & \p \\
\p & \p & \p & \p & \p & \p & \p \\
\p & \p & \p & \p & \p & \p & \p \\
\p & 1 & \p & \p & \p & \p & \p \\
\p & \p & \p & \p & \p & \p & \p \\
\p & \p & \p & \p & \p & \p & \p \\
\end{smallmatrix}
\right)_{\hspace{-5pt} C}
$$
\end{itemize}

\section*{Appendix A.2: Classification of graded contractions }
The lists of all contracted Lie algebras of the Gell--Mann
graded $\slc$ is presented.
Indecomposable non--abelian parts of the contracted Lie algebras
are tabulated. Algebras are divided into classes according to
the dimensions of the derived series (DS), the lower central series
(CS) and the upper central series (US). For each of the listed Lie
algebras we give its nonzero commutation relations, dimensions of
algebras of generalized derivations $\dima$, number of formal
invariants $\tau$ and the type of the contraction (C--continuous,
D--discrete). The Levi decomposition is given in the last column for
non--solvable Lie algebras. For the non--solvable and the solvable
non--nilpotent Lie algebras the nilradical is added. Casimir operators of the
nilpotent Lie algebras are presented.

For specification of the ranges of parameters for one--parametric
contractions, the following notations are used
$$
\renewcommand{\arraystretch}{1.5}
\begin{array}{l}
\bC_{10} = \set{z\in \bC}{0<|z|<1}\cup \set{z\in\bC}{|z|=1 \wedge
\operatorname{Im}(z)\geq 0},\\
\bC_{20} = \set{z\in\bC}{0<|z+1|<1 \wedge \operatorname{Re}(z)
\geq -\frac{1}{2}}\cup  \set{z\in\bC}{|z+1|=1 \wedge \operatorname{Re}(z)\geq
-\frac{1}{2} \wedge \operatorname{Im}(z)
> 0}.
\end{array}
$$
We use the superscript $\ast$ for any of the listed sets if there
are no isomorphisms among Lie algebras corresponding to different
parameters in the given set.

For low--dimensional Lie algebras the alternative name (AN) is
assigned according to the list of algebras from \cite{PSWZ_Inv}.
This name is also used for nilradicals and Levi decompositions. In order to display the structure of the contracted Lie algebras, their commutation relations are written in the basis $(e_1,e_2,\dots,e_n)$ which for each case begins with the vectors from the center $(e_1,\dots,e_k)$ and extends to the vectors from the nilradical $(e_1,\dots,e_l)$ and the radical $(e_1,\dots,e_m)$,  $0 \leq k  \leq l\leq m \leq n.$
{\fontsize{10pt}{1}
\begin{sidewaystable}
\vspace{350pt}
\tabcolsep 3pt \vspace{52pt} 
\begin{tabular}{lllllccc}
\multicolumn{8}{l}{\normalsize \bf Non--solvable graded contractions} \\
\hline DS, CS, US & Name & Commutation relations & $\dima$ & $\tau$ & Nilradical & T & Levi dec.\\

\hline$(3)(3)(0)$ & $\mG'_{18,8}$ & $[e_1,e_2]=e_3,\ [e_1,e_3]=e_2,\ [e_2,e_3]=e_1$ & $[3,0,1,0,0,1]$ & 1 & $\{0\}$ & $D$ & $\sla$\\
\hline$(6)(6)(0)$ & $\mG'_{12,2}$& $[e_1,e_4]=e_2,\ [e_1,e_6]=-e_3,\ [e_2,e_4]=-e_1,\ [e_2,e_5]=-e_3,$ & $[7,0,2,0,0,2]$ & 2 & $3\mA_1$ & $D$ & $3\mA_1\triangleleft \sla$\\
 & & $[e_3,e_5]=e_2,\ [e_3,e_6]=e_1,\ [e_4,e_5]=e_6,\ [e_4,e_6]=-e_5,\ [e_5,e_6]=e_4$ \\
\hline$(8)(8)(0)$ & $\mG_{6,2}$& $[e_1,e_7]=2e_1,\ [e_1,e_8]=\sqrt{2}e_2,\ [e_2,e_6]=-\sqrt{2}e_1,\ [e_2, e_7]=e_2,$ & $[9,0,1,0,0,1]$ & 2 & $5\mA_1$ & $C$ & $5\mA_1 \triangleleft \sla$ \\
 & &  $[e_2,e_8]=\sqrt{3}e_3,\ [e_3,e_6]=-\sqrt{3}e_2,\ [e_3,e_8]=\sqrt{3}e_4,\ [e_4,e_6]=-\sqrt{3}e_3,$ \\
 & & $[e_4,e_7]=-e_4,\ [e_4,e_8]=\sqrt{2}e_5,\ [e_5,e_6]=-\sqrt{2}e_4,\ [e_5,e_7]=-2e_5,$ \\
 & &  $[e_6,e_7]=e_6,\ [e_6,e_8]=e_7,\ [e_7,e_8]=e_8 $  \\
\hline$(87)(87)(0)$ & $\mG_{6,1}$& $[e_1,e_5]=e_1,\ [e_1,e_6]=e_2,\ [e_1,e_8]=e_1,\ [e_2,e_5]=e_2,\ [e_2,e_7]=e_1, $ & $[9,0,1,0,0,1]$ & 2 & $4\mA_1$ & $C $ & $\mA_{5,7}^{(1,-1,-1)}\triangleleft \sla$\\
 & & $[e_2,e_8]=-e_2,\ [e_3,e_5]=-e_3,\ [e_3, e_6]=e_4,\ [e_3,e_8]=e_3,$ \\
 & & $[e_4,e_5]=-e_4,\ [e_4,e_7]=e_3,\ [e_4,e_8]=-e_4,\ [e_6,e_7]=e_8,$\\
 & & $[e_6,e_8]=-2e_6,\ [e_7,e_8]=2e_7$ \\
\hline
\end{tabular}

\vspace{50pt}

\begin{tabular}{llllllll}
\multicolumn{7}{l}{\normalsize \bf  Nilpotent graded contractions} \\
\hline DS, CS, US & Name & Commutation relations & $\dima$ & $\tau$ & Casimir operators & T & AN\\
\hline $(310) (310) (13)$ & $\mG'_{20,1}$ & $[e_2,e_3]=e_1$ & $[6,6,3,5,3,4]$ & $1$ & $e_1$ & C & $\mA_{3,1} $\\
\hline $(510) (510) (15)$ & $\mG'_{19,4}$ & $[e_2,e_4]=e_1,\ [e_3,e_5]=e_1$ & $[15,15,5,14,10,11]$  & $1$ & $e_1$ & C & $\mA_{5,4} $\\
\hline $(520) (520) (25)$ & $\mG'_{19,3}$ & $[e_3,e_5]=e_1,\ [e_4,e_5]=e_2$ & $[13,13,7,9,7,11]$ & $3$ & $e_1,\ e_2,\ e_2e_3-e_1e_4$ &C & $\mA_{5,1} $\\
\hline $(620) (620) (26)$ & $\mG'_{18,7}$ & $[e_3,e_5]=e_1,\ [e_4,e_6]=e_1,\ [e_5,e_6]=e_2$ & $[17,18,10,14,10,14]$ & $2$ & $e_1,\ e_2$ & C & $\mA_{6,4} $\\
\hline $(630) (630) (36)$ & $\mG'_{18,16}$ & $[e_4,e_5]=e_1,\ [e_4,e_6]=e_2,\ [e_5,e_6]=e_3$ & $[18,18,10,9,10,19]$ & $4$ & $e_1,\ e_2,\ e_3,\ e_1e_6-e_2e_5+e_3e_4$ & C & $\mA_{6,3} $\\
\hline $(630) (6310) (136)$ & $\mG'_{17,13}$ & $[e_2,e_5]=e_1,\ [e_3,e_6]=e_1,\ [e_4,e_5]=e_3,\ [e_4,e_6]=e_2$ & $[11,10,4,6,4,8]$ & $2$ & $e_1,\ e_2e_3-e_1e_4$ & D & $\mA_{6,14}^{(1)} $\\
\hline $(710) (710) (17)$ & $\mG'_{18,18}$ & $[e_2,e_5]=e_1,\ [e_3,e_6]=e_1,\ [e_4,e_7]=e_1$ & $[28,28,7,27,21,22]$ & $1$ & $e_1$ & C  \\
\hline $(720) (720) (27)$ & $\mG'_{17,5}$ & $[e_3,e_6]=e_1,\ [e_4,e_7]=e_2,\ [e_5,e_6]=e_2,\ [e_5,e_7]=e_1$ & $[19,19,11,15,11,15]$ & $3$ & $e_1,\ e_2,\ e_1^2e_4-e_1e_2e_5+e_2^2e_3$ & C  \\
 & $\mG'_{18,5}$ & $[e_3,e_6]=e_1,\ [e_4,e_7]=e_1,\ [e_5,e_7]=e_2$ & $[21,22,11,18,14,18]$ & $3$ & $e_1,\ e_2,\ e_1e_5-e_2e_4$ & C  \\
\hline
\end{tabular}
\end{sidewaystable}

\begin{sidewaystable}
\vspace{400pt}
\tabcolsep 4pt \vspace{43pt} 
\begin{tabular}{lllllll}
\multicolumn{7}{l}{\normalsize \bf Nilpotent graded contractions (continued)} \\
\hline DS, CS, US & Name & Commutation relations & $\dima$ & $\tau$ & Casimir operators & T \\
\hline $(730) (730) (37)$ & $\mG'_{17,16}$ & $[e_4,e_6]=e_1,\ [e_4,e_7]=e_2,\ [e_5,e_6]=e_2,\ [e_5,e_7]=e_3$  & $[19,24,13,15,13,22]$ & $3$ & $e_1,\ e_2,\ e_3$ & C  \\
 & $\mG'_{16,12}$ & $[e_4,e_6]=e_1,\ [e_5,e_7]=e_2,\ [e_6,e_7]=e_3$ & $[20,24,13,15,13,22]$ & $3$ & $e_1,\ e_2,\ e_3$ & C  \\
 & $\mG'_{17,12}$ & $[e_4,e_7]=e_1,\ [e_5,e_6]=e_1,\ [e_5,e_7]=e_2,\ [e_6,e_7]=e_3$ & $[22,24,13,15,13,22]$ & $3$ & $e_1,\ e_2,\ e_3$ & C  \\
 & $\mG'_{18,6}$ & $[e_4,e_7]=e_1,\ [e_5,e_7]=e_2,\ [e_6,e_7]=e_3$ & $[25,25,13,16,13,22]$ & $5$ & $e_1,\ e_2,\ e_3,\ e_1e_5-e_2e_4,\ e_1e_6-e_3e_4 $ & C  \\
\hline $(730) (7310) (147)$ & $\mG'_{16,8}$ & $[e_2,e_5]=e_1,\ [e_3,e_6]=e_1,\ [e_4,e_7]=e_1,\ [e_5,e_7]=e_3,$ & $[15,16,7,10,7,11]$ & $1$ & $e_1$ & D  \\
 & & $[e_6,e_7]=e_2$\\
\hline $(740) (7410) (147)$ & $\mG'_{15,6}(a)$ & $[e_2,e_5]=(a+1)e_1,\ [e_3,e_6]=e_1,\ [e_4,e_7]=e_1,$ & $[15,13,7,9,6,11]$ & $1$ & $e_1$ & D  \\
 & & $[e_5,e_6]=-ae_4,\ [e_5,e_7]=e_3,\ [e_6,e_7]=e_2 \quad a\in\bC_{20}^\ast$ \\
 & & $a=-\frac{1}{2} \cong a=1$ &  $[17,13,7,9,6,11]$ & & & C\\
\hline $(740) (7410) (247)$ & $\mG'_{16,7}$ & $[e_3,e_6]=e_1,\ [e_4,e_7]=e_1,\ [e_5,e_6]=e_4,\ [e_5,e_7]=e_3,$ & $[15,17,8,9,7,16]$ & $3$ & $e_1,\ e_2,\ e_1e_5-e_3e_4$ & D \\
 & & $[e_6,e_7]=e_2$\\
\hline $(820) (820) (28)$ & $\mG_{18,23}$ & $[e_3,e_6]=e_1,\ [e_4,e_7]=e_2,\ [e_5,e_8]=e_1+e_2$ & $[22,25,13,21,15,19]$ & $2$ & $e_1,\ e_2$ & C  \\
\hline $(830) (830) (38)$ & $\mG_{15,3}$ & $[e_4,e_6]=e_1,\ [e_4,e_8]=e_2,\ [e_5,e_7]=e_3,\ [e_5,e_8]=e_2,$ & $[19,24,16,15,16,25]$ & $4$ & $e_1,\ e_2,\ e_3,\ e_1(e_1e_5-e_2e_7+e_3e_8)$ & C \\
 & & $[e_6,e_7]=e_2,\ [e_6,e_8]=e_3,\ [e_7,e_8]=e_1$ & & & $+e_2^2(e_5-e_4)-e_2e_3e_6+e_3^2e_4$ \\
 & $\mG_{16,2}$ & $[e_4,e_6]=e_1,\ [e_4,e_8]=e_2,\ [e_5,e_7]=e_3,\ [e_5,e_8]=e_2,$ & $[20,24,16,15,16,25]$ & $4$ & $e_1,\ e_2,\ e_3,\ e_1(e_2e_7-e_3e_8) $& C \\
 & & $[e_6,e_7]=e_2,\ [e_6,e_8]=e_3$ & & & $+e_2^2(e_4-e_5)+e_2e_3e_6-e_3^2e_4$ \\
 & $\mG_{17,3}$ & $[e_4,e_6]=e_1,\ [e_5,e_7]=e_2,\ [e_5,e_8]=e_3,\ [e_6,e_8]=e_3,$ & $[21,25,16,16,16,25]$ & $4$ & $e_1,\ e_2,\ e_3,$ & C \\
 & & $[e_7,e_8]=e_1$ & & & $e_1(e_1e_5+e_2e_8-e_3e_7)+e_2e_3e_4$ \\
 & $\mG_{18,1}$ & $[e_4,e_7]=e_1,\ [e_5,e_7]=e_2,\ [e_5,e_8]=e_2,\ [e_6,e_8]=e_3$ & $[22,28,16,19,16,25]$ & $4$ & $e_1,\ e_2,\ e_3,\ e_1(e_2e_6-e_3e_5)+e_2e_3e_4$ & C \\
 & $\mG_{16,6}$ & $[e_4,e_7]=e_1,\ [e_5,e_8]=e_2,\ [e_6,e_7]=e_2,\ [e_6,e_8]=e_1,$ & $[26,28,16,19,16,25]$ & $4$ & $e_1,\ e_2,\ e_3,\ e_1^2e_5-e_1e_2e_6+e_2^2e_4$ & C \\
 & & $[e_7,e_8]=e_3$\\
 & $\mG_{17,4}$ & $[e_4,e_7]=e_1,\ [e_5,e_8]=e_2,\ [e_6,e_8]=e_1,\ [e_7,e_8]=e_3$ & $[27,30,17,21,17,26]$ & $4$ & $e_1,\ e_2,\ e_3,\ e_1e_5-e_2e_6$ & C \\
\hline $(840) (840) (48)$ & $\mG_{15,9}$ & $[e_5,e_7]=e_1,\ [e_5,e_8]=e_2,\ [e_6,e_7]=e_3,\ [e_6,e_8]=e_4$ & $[24,33,17,17,17,33]$ & $4$ & $e_1,\ e_2,\ e_3,\ e_4$ & C \\
 & $\mG_{16,9}$ & $[e_5,e_7]=e_1,\ [e_5,e_8]=e_2,\ [e_6,e_7]=e_2,\ [e_6,e_8]=e_3,$ & $[25,33,17,17,17,33]$ & $4$ & $e_1,\ e_2,\ e_3,\ e_4$ & C \\
 & & $[e_7,e_8]=e_4$\\
 & $\mG_{17,15}$ & $[e_5,e_7]=e_1,\ [e_6,e_7]=e_2,\ [e_6,e_8]=e_3,\ [e_7,e_8]=e_4$ & $[27,33,17,17,17,33]$ & $4$ & $e_1,\ e_2,\ e_3,\ e_4$ & C \\
\hline $(840) (8410) (148)$ & $\mG_{15,4}$ & $[e_2,e_6]=e_1,\ [e_3,e_7]=e_1,\ [e_4,e_8]=e_1,\ [e_5,e_7]=e_4,$ & $[18,16,6,10,7,12]$ & $2$ & $e_1,\ e_1e_5-e_3e_4$ & C \\
 & & $[e_5,e_8]=e_3,\ [e_6,e_8]=e_3,\ [e_7,e_8]=e_2$ \\
\hline $(850) (8520) (258)$ & $\mG_{15,5}$ & $[e_3,e_6]=e_1,\ [e_4,e_7]=e_1+e_2,\ [e_5,e_8]=e_2,$ & $[18,19,7,9,6,17]$ & $2$ & $e_1,\ e_2$ & C \\
    & & $[e_6,e_7]=e_5,\ [e_6,e_8]=e_4,\ [e_7,e_8]=e_3$  \\
\hline
\end{tabular}
\end{sidewaystable} 

\begin{sidewaystable}
\vspace{450pt}
\tabcolsep 3pt \vspace{12pt} 
\begin{tabular}{lllllcc}
\multicolumn{7}{l}{\normalsize \bf Solvable non--nilpotent graded contractions } \\
\hline DS, CS, US & Name & Commutation relations \hspace{200pt} AN & $\dima$ & $\tau$ & Nilradical & T \\
\hline $(320) (32) (0)$ & $\mG'_{19,2}$ & $[e_1,e_3]=e_1,\ [e_2,e_3]=-e_2$ \hspace{200pt} $\mA_{3,4}$ & $[4,3,1,2,0,1]$ & 1 & $2\mA_1$ & $D$ \\
\hline $(530) (532) (12)$ & $\mG'_{18,9}$ & $[e_2,e_5]=e_1,\ [e_3,e_5]=e_3,\ [e_4,e_5]=-e_4$ \hspace{130pt} $\mA_{5,8}^{(-1)}$ & $[8,9,3,5,2,6]$ & 3 & $4\mA_1$ & $D$ \\
\hline $(540) (54) (0)$ & $\mG'_{17,11}(a)$ & $[e_1,e_5]=e_2,\ [e_2,e_5]=e_1,\ [e_3,e_5]=e_4,\ [e_4,e_5]=ae_3,\ a\in \bC_{10}^\ast$ \hspace{30pt} $\mA_{5,17}^{(\sqrt{a},0,0)}$ & $[8,5,1,4,0,1]$ & 3 & $4\mA_1$ & $D$ \\
 &  & $a=1$ & $[12,5,1,4,0,1]$\\
\hline $(640) (64) (0)$ & $\mG'_{15,7}$ & $[e_1,e_6]=e_2,\ [e_2,e_6]=e_1,\ [e_3,e_5]=e_1,\ [e_3,e_6]=e_4,\ [e_4,e_5]=e_2,\ [e_4,e_6]=e_3$ & $[9,6,2,4,0,2]$ & 2 & $\mA_{5,1} $ & $D$  \\
\hline $(6510) (65) (1)$ & $\mG'_{15,8}$ & $[e_2,e_4]=e_1,\ [e_2,e_6]=e_5,\ [e_3,e_5]=e_1,\ [e_3,e_6]=e_4,\ [e_4,e_6]=e_3,\ [e_5,e_6]=e_2$ & $[10,6,2,1,1,7]$ & 2 & $\mA_{5,4}$ & $D$ \\
\hline $(740) (742) (24)$ & $\mG'_{17,6}$ & $[e_3,e_7]=e_1,\ [e_4,e_7]=e_2,\ [e_5,e_7]=e_6,\ [e_6,e_7]=e_5$ & $[16,19,7,10,6,15]$  & 5 & $6\mA_1 $ & $D$ \\
\hline $(750) (754) (1)$ & $\mG'_{12,1}$ & $[e_2,e_6]=e_3,\ [e_2,e_7]=e_4,\ [e_3,e_6]=e_2,\ [e_3,e_7]=e_5,\ [e_4,e_6]=e_5,\ [e_4,e_7]=e_2,$ & $[10,12,3,6,2,8]$ & 3 & $5\mA_1$ & $D$ \\
 & & $[e_5,e_6]=e_4,\ [e_5,e_7]=e_3,\ [e_6,e_7]=e_1$ \\
 & $\mG'_{14,1}$ & $[e_2,e_7]=e_3,\ [e_3,e_7]=e_2,\ [e_4,e_6]=e_2,\ [e_4,e_7]=e_5,\ [e_5,e_6]=e_3,\ [e_5,e_7]=e_4,$ & $[11,12,3,6,2,8]$ & 3 & $\mA_1 \oplus \mA_{5,1}$ & $D$ \\
 & & $[e_6,e_7]=e_1$\\
\hline $(750) (754) (12)$ & $\mG'_{16,4}(a)$ & $[e_2,e_7]=ae_3,\ [e_3,e_7]=e_2,\ [e_4,e_7]=e_5,\ [e_5,e_7]=e_4,\ [e_6,e_7]=e_1,\ a\in\bC_{10}^\ast$ & $[12,13,3,7,2,8]$ & 5 & $6\mA_1$ & $D$  \\
 & & $a=1$ & $[16,13,3,7,2,8]$\\
\hline $(7510) (75) (12)$ & $\mG'_{14,2}$ & $[e_2,e_4]=e_1,\ [e_2,e_7]=e_5,\ [e_3,e_5]=e_1,\ [e_3,e_7]=e_4,\ [e_4,e_7]=e_3,\ [e_5,e_7]=e_2,$ & $[12,12,3,3,2,8]$ & $1$ & $\mA_1 \oplus \mA_{5,4}$ & $D$  \\
 & & $[e_6,e_7]=e_1$\\
\hline $(760) (76) (0)$ & $\mG'_{15,2}(a,b)$ & $[e_1,e_7]=4ae_4,\ [e_2,e_7]=be_5,\ [e_3,e_7]=e_6,\ [e_4,e_7]=e_1,\ [e_5,e_7]=e_2,$ & $[12,7,1,6,0,1]$ & $5$ & $6\mA_1$ & $D$  \\
 & & $[e_6,e_7]=e_3,\ a,b\neq 0$\\
 & & $a=\frac{1}{4}\ \xor\ b=1 \ \xor\ b=4a$ & $[16,7,1,6,0,1]$\\
 & & $a=\frac{1}{4} \wedge  b=1$ & $[24,7,1,6,0,1]$\\
 & & $a=b=1$ & & & & $C$\\
\hline $(840) (842) (25)$ & $\mG_{17,1}$ & $[e_3,e_5]=e_1,\ [e_3,e_8]=-e_1,\ [e_4,e_5]=e_2,\ [e_4,e_8]=e_2,\ [e_6,e_8]=e_7,\ [e_7,e_8]=e_6$ & $[16,21,9,12,8,17]$ & $4$ & $2\mA_1 \oplus \mA_{5,1}$  & $D$ \\
\hline $(850) (854) (12)$ & $\mG_{16,1}(a)$ & $[e_2,e_8]=ae_3,\ [e_3,e_8]=e_2,\ [e_4,e_7]=e_1,\ [e_4,e_8]=e_1,\ [e_5,e_7]=e_6,$ & $[13,14,4,8,3,9]$ & $4$ & $6\mA_1$  & $D $ \\
 & & $[e_6,e_7]=e_5,\ a\in \bC_{10}$\\
\hline $(860) (86) (0)$ & $\mG_{15,1}(a,b)$ & $[e_1,e_8]=ae_2,\ [e_2,e_8]=e_1,\ [e_3,e_7]=be_4,\ [e_4,e_7]=e_3,\ [e_5,e_7]=e_6,\ [e_5,e_8]=e_6,$ & $[12,7,1,6,0,1]$ & $4$ & $6\mA_1$  & $D $  \\
 & & $[e_6,e_7]=e_5,\ [e_6,e_8]=e_5,\ a,b \neq 0$ \\
 & & $a=b=1$ & & & & $C$\\
\hline $(8620) (86) (0)$ & $\mG_{11,2}$ & $[e_1,e_7]=e_1,\ [e_1,e_8]=e_1,\ [e_2,e_3]=e_1,\ [e_2,e_7]=e_2,\ [e_3,e_8]=e_3,\ [e_4,e_7]=-e_4,$ & $[10,2,1,2,0,1]$ & $2$ & $\mA_{3,1} \oplus \mA_{3,1}$  & $C$ \\
& & $[e_4,e_8]=-e_4,\ [e_5,e_6]=e_4,\ [e_5,e_7]=-e_5,\ [e_6,e_8]=-e_6 $ \\
& $\mG_{11,1}$ & $[e_1,e_8]=e_2,\ [e_2,e_8]=e_1,\ [e_3,e_6]=e_1,\ [e_3,e_8]=e_5,\ [e_4,e_7]=e_2,\ [e_4,e_8]=e_5,$ & $[11,7,2,2,0,2]$ & $2$ & $\mG'_{17,5}$ & $C$ \\
& & $[e_5,e_6]=e_2,\ [e_5,e_7]=e_1,\ [e_5,e_8]=2e_3+2e_4,\ [e_6,e_8]=-e_7,\ [e_7,e_8]=-e_6 $ \\
\hline $(8730) (87) (1)$ & $\mG_{11,3}$ & $[e_2,e_8]=2e_3,\ [e_3,e_8]=2e_2,\ [e_4,e_6]=e_1+e_3,\ [e_4,e_7]=e_2,\ [e_4,e_8]=e_5,$ & $[12,10,2,3,1,9]$ & $2$ & $\mG'_{17,16} $ & $C$\\
& & $[e_5,e_6]=e_2,\ [e_5,e_7]=e_3-e_1,\ [e_5,e_8]=e_4,\ [e_6,e_8]=e_7,\ [e_7,e_8]=e_6$  \\
\hline
\end{tabular}
\end{sidewaystable}
}

\section*{Appendix A.3: Invariant Functions of one--parametric graded contractions
}
The blank space in the table stands for
general complex number different from all previously listed in
given table.
{\leftmargini 0pt
\begin{itemize}
\item[$\bullet$\hspace{-10pt}]\hspace{4pt} $\mG'_{17,11}(a),\ a
\neq 0$, \qquad $\mG'_{17,11}(a) \cong
\mG'_{17,11}\left(\frac{1}{a}\right)$ \qquad $\longrightarrow$ \
$a\in\bC^\ast_{10}$
\vspace{-5pt}
\begin{center}
\begin{tabular}[t]{|l||c|c|c|}
\multicolumn{4}{l}{$a=1$} \\
\hline \parbox[l][18pt][c]{0pt}{} $\alpha$ & -1 & 1 &   \\
\hline \parbox[l][18pt][c]{0pt}{} $\psi(\alpha)$ & 13 & 12 & 5 \\
\hline
\end{tabular}\quad
\begin{tabular}[t]{|l||c|c|c|c|c|}
\multicolumn{6}{l}{$a=-1$} \\
\hline \parbox[l][18pt][c]{0pt}{} $\alpha$ & $-i$ & $i$ & -1 & 1 &  \\
\hline \parbox[l][18pt][c]{0pt}{} $\psi(\alpha)$ & 9 & 9 & 9 & 8 & 5\\
\hline
\end{tabular}\quad
\begin{tabular}[t]{|l||c|c|c|c|c|c|c|}
\multicolumn{8}{l}{$a \neq 0,\pm 1$} \\
\hline \parbox[l][18pt][c]{0pt}{} $\alpha$ & -1 & 1 & $\frac{1}{\sqrt{a}}$ & $-\frac{1}{\sqrt{a}}$ & $\sqrt{a}$ & $-\sqrt{a}$ & \\
\hline \parbox[l][18pt][c]{0pt}{} $\psi(\alpha)$ & 9 & 8 & 7 & 7 & 7 & 7 & 5  \\
\hline
\end{tabular}
\end{center}
\vspace{0pt}
\item[$\bullet$\hspace{-10pt}]\hspace{4pt} $\mG'_{16,4}(a),\ a
\neq 0$, \qquad $\mG'_{16,4}(a) \cong
\mG'_{16,4}\left(\frac{1}{a}\right)$ \qquad $\longrightarrow$ \
$a\in\bC^\ast_{10}$

\begin{center}
\tabcolsep 3pt
\begin{tabular}[t]{|l||c|c|c|c|}
\multicolumn{5}{l}{$a=1$} \\
\hline \parbox[l][18pt][c]{0pt}{} $\alpha$ & $-1$ & 1 & 0 & \\
\hline \parbox[l][18pt][c]{0pt}{} $\psi(\alpha)$ & 17 & 16 & 13 & 9\\
\hline
\end{tabular}\quad 
\begin{tabular}[t]{|l||c|c|c|c|c|c|}
\multicolumn{7}{l}{$a=-1$} \\
\hline \parbox[l][18pt][c]{0pt}{} $\alpha$ & 0 & $-i$ & $i$ & $-1$ & 1 &  \\
\hline \parbox[l][18pt][c]{0pt}{} $\psi(\alpha)$ & 13 & 13 &  13 & 13 & 12 & 9\\
\hline
\end{tabular}\quad
\begin{tabular}[t]{|l||c|c|c|c|c|c|c|c|}
\multicolumn{9}{l}{$a \neq 0,\pm 1$} \\
\hline \parbox[l][18pt][c]{0pt}{} $\alpha$ & 0 & $-1$ & 1 & $\frac{1}{\sqrt{a}}$ & $-\frac{1}{\sqrt{a}}$ & $\sqrt{a}$ & $-\sqrt{a}$ & \\
\hline \parbox[l][18pt][c]{0pt}{} $\psi(\alpha)$ & 13 & 13 & 12 & 11 & 11 & 11 & 11 & 9 \\
\hline
\end{tabular}\
\end{center}
\vspace{0pt}

\item[$\bullet$\hspace{-10pt}]\hspace{4pt} $\mG_{16,1}(a),\ a \neq
0$, \qquad $\mG'_{16,1}(a)\cong \mG'_{16,1}\left(\frac{1}{a}\right)$
\qquad $\longrightarrow$ \ $a\in\bC_{10}$

\item[$\bullet$\hspace{-10pt}]\hspace{4pt} $\mG'_{15,6}(a),\ a
\neq 0,-1$, \qquad $\longrightarrow$ \ $a\in\bC^\ast_{20}$
$$\mG'_{15,6}(a) \cong \mG'_{15,6}\left(\frac{1}{a}\right) \cong
\mG'_{15,6}(-a-1)\cong \mG'_{15,6}\left(\frac{-1}{a+1}\right) \cong
\mG'_{15,6}\left(\frac{-a}{a+1}\right) \cong
\mG'_{15,6}\left(\frac{a+1}{-a}\right)$$

\begin{center}
\tabcolsep 2pt
\begin{tabular}[t]{|l||c|c|c|c|}
\multicolumn{5}{l}{$a = -\frac{1}{2},-2,1$} \\
\hline \parbox[l][18pt][c]{0pt}{} $\alpha$ & 1 & $-\frac{1}{2}$ & $-2$ &  \\
\hline \parbox[l][18pt][c]{0pt}{} $\psi(\alpha)$ & 17 & 15 & 15 & 13 \\
\hline
\end{tabular}\
\begin{tabular}[t]{|l||c|c|c|c|}
\multicolumn{5}{l}{$a = -\frac{1 \pm \sqrt{3}i}{2}$} \\
\hline \parbox[l][18pt][c]{0pt}{} $\alpha$ & $-\frac{1 + \sqrt{3}i}{2}$ & $-\frac{1 - \sqrt{3}i}{2}$ & 1 &  \\
\hline \parbox[l][18pt][c]{0pt}{} $\psi(\alpha)$ & 16 & 16 & 15 & 13\\
\hline
\end{tabular}\
\begin{tabular}[t]{|l||c|c|c|c|c|c|c|c|}
\multicolumn{9}{l}{$a \neq 0,\pm 1,-\frac{1}{2},-2, -\frac{1 \pm \sqrt{3}i}{2}$} \\
\hline \parbox[l][18pt][c]{0pt}{} $\alpha$ & 1 & $a$ & $\frac{1}{a}$ & $-1-a$ & $-\frac{1}{1+a}$ & $-\frac{a}{1+a}$ & $-\frac{1+a}{a}$ & \\
\hline \parbox[l][18pt][c]{0pt}{} $\psi(\alpha)$ & 15 & 14 & 14 & 14 & 14 & 14 & 14 & 13\\
\hline
\end{tabular}
\end{center}
\end{itemize}
}


\end{document}